\documentclass[11pt]{amsart}
\usepackage{geometry}                
\usepackage{graphicx}
\usepackage{amssymb}
\usepackage{color}
\newcommand{\BBr}{\mathbb{R}}

\newcommand{\BBe}{\mathbb{E}}
\renewcommand{\P}{\mathbb{P}}
\newcommand{\Ea}{\mathbb{E}_\alpha}
\DeclareMathOperator{\dom}{dom}
\newtheorem{theorem}{Theorem}[section]
\newtheorem{lemma}{Lemma}[section]
\newtheorem{proposition}{Proposition}[section]
\newtheorem{remark}{Remark}[section]

\newtheorem{definition}{Definition}[section]
\title{Bregman superquantiles. Estimation methods and applications}
\author{T. Labopin-Richard \and F. Gamboa \and A. Garivier \and B. Iooss}
\address{FG AG TLR  and BI are with the Institut de Math\'ematiques de Toulouse (CNRS UMR 5219). Universit\'e Paul Sabatier, 118 route de Narbonne, 31062 Toulouse, France. BI is also with EDF R\&D.}

\date{\today}                                           

\begin{document}
\maketitle

\begin{abstract}
In this work, we extend some parameters built on a probability distribution introduced before to the case where the proximity between real numbers is measured by using a Bregman divergence. This leads to the definition of the Bregman superquantile (that we can connect with several works in economy, see for example \cite{eco1} or \cite{gini}). Axioms of a coherent measure of risk discussed previously (see \cite{def} or \cite{coherent}) are studied in the case of Bregman superquantile. Furthermore, we deal with asymptotic properties of a Monte Carlo estimator of the Bregman superquantile.
Several numerical tests confirm the theoretical results and an application illustrates the potential interests of the Bregman superquantile.
\end{abstract}

\section{Introduction}

\subsection{Aim and scope}

The aim of this article is to define and to study properties and estimation procedures for Bregman extension of the superquantile defined in \cite{rocury00} or in \cite{ref24} (see also \cite{rocroy13}, \cite{rocroy13a} and references therein). In the introduction we first recall the necessary conditions for a measure of risk to be coherent. Further in Section 2 we present the superquantile as a partial response to this problem. We also introduce the Bregman superquantile and study axioms of a coherent measure of risk for this quantity. In Section 3 we seek to estimate this Bregman superquantile, we introduce a plug-in estimator and study its convergence and its asymptotic normality. Some numerical simulations are shown in Section 4. An application on real data of radiological exposure is given in Section 5. All the proofs are postponed to Section 6.

\subsection{Coherent measures of risk}

Let $X$ be a real-valued random variable and let $F_X$ be its cumulative distribution function. We define for $u\in]0,1[$, the quantile function

$$F^{-1}_X(u):=\inf\{x: F_X(x)\geq u\}.$$

A usual way to quantify the risk associated with $X$ is to consider, for a given number $\alpha\in]0,1[$ close to 1, its lower quantile $q_\alpha^X:=F^{-1}_X(\alpha)$.

Nevertheless, the  quantile is not a subadditive function of $X$, a major property in some applications (e.g. finance, see \cite{coherent}). Moreover, the quantile does not give any information about what is happening in the distribution-tail above the quantile (which can be dangerous when we deal with insurance premiums for example). In \cite{def} a new quantity called superquantile satisfying the property of subadditivity and giving more information about the distribution-tail is introduced. The superquantile is defined by  

$$Q_{\alpha}:=Q_{\alpha}(X)= \mathbb{E}(X | X \geq q_{\alpha}^X)=\mathbb{E}(X | X \geq F_X^{-1}(\alpha))$$

We can notice that $Q^{\alpha}$ is always well defined as an element of $\bar{\mathbb{R}}= \mathbb{R} \cup \{ + \infty\}$. Indeed, if the expectation is not finite we may set it to $+ \infty$. We indeed have 

$$Q_{\alpha} = \frac{\mathbb{E} \left(X \mathbf{1}_{X \geq F_X^{-1}(\alpha)} \right)}{\mathbb{P}\left(X \geq F_X^{-1}(\alpha)\right)} \geq  F_X^{-1}(\alpha) \frac{\mathbb{E}\left(\mathbf{1}_{X \geq  F_X^{-1}(\alpha)}\right)}{\mathbb{P}\left(X \geq  F_X^{-1}(\alpha)\right)}= F_X^{-1}(\alpha).$$

\begin{remark}
In particular, when $F_X\left(F_X^{-1} (\alpha)\right)=\alpha$, the Bayes formula gives us

$$Q_{\alpha} =\mathbb{E} \left(X | X \geq F_X^{-1}(\alpha)\right)= \mathbb{E}\left(\frac{X \mathbf{1}_{ X \geq F_X^{-1}(\alpha)}}{1-\alpha}\right).$$

\end{remark}

\vspace{0.3cm}

From now, we always deal with random variables $X$ which distribution are continuous (that is $F_X$ is continuous and then $F_X\left(F_X^{-1} (\alpha)\right)=\alpha$). 

\vspace{0.3cm}

Notice that this quantity is also called conditional value at risk in other references (\cite{rocury00}, \cite{rocury13}, \cite{rocroy13a}). Further, (when $F$ is conitnuous), it is also a distortion measure of risk studied for example in \cite{ref1}, \cite{ref13}, \cite{ref14}, \cite{ref15}, \cite{ref16}, \cite{ref24}, \cite{ref25}, \cite{ref26} and \cite{lui}. In these papers, a distortion measure of risk is the quantity

$$\rho_g(X) := - \int_{\mathbb{R}} x dg\left( F_X(x) \right),$$

where $g$, called the distortion function, is a map from $[0, 1]$ to $[0, 1]$. $g$ is assumed to be nondecreasing and such that $g(0)=0$ and $g(1)=1$. Then, taking $$g(x)=\alpha \left( \dfrac{x}{\alpha} \wedge 1\right),$$ we get 

$$\rho_g(X)= \alpha Q_{1-\alpha}(-X).$$


\vspace{0.3cm}

Sub-additivity is not the sole interesting property for a measure of risk (for example for financial applications). Following \cite{def} we define:

\begin{definition}
\label{deco}
Let $\mathcal{R}$ be a measure of risk that is a numerical function defined on random variables. Let $X$ and $X'$ be two real-valued random variables.
We say that $\mathcal{R}$ is coherent if, and only if, it satisfies the five following properties : 
\begin{itemize}
\item[i)] Constant invariance : let $C\in\BBr$, if $X=C$ (a.s.)  then $\mathcal{R}(C)=C$.
\item[ii)] Positive homogeneity :  $\forall \lambda > 0$, $\mathcal{R}(\lambda X) = \lambda \mathcal{R}(X)$.
\item[iii)] Subaddidivity : $\mathcal{R}(X + X') \leq \mathcal{R}(X) + \mathcal{R}(X')$.
\item[iv)] Monotonicity : If $X \leq X'$ (a.s.) then $\mathcal{R}(X) \leq \mathcal{R}(X')$.
\item[v)]  Closeness : Let $(X_h)_{h\in\BBr}$ be a collection of random variables.\\
If $\mathcal{R}(X_h) \leq 0$ and ${\displaystyle \lim_{h\rightarrow 0}||X_{h}-X||_{2}=0}$ then $\mathcal{R}(X) \leq 0$.
\end{itemize}
\end{definition}

The superquantile is a coherent measure of risk (see \cite{archiv}, \cite{croissance}, \cite{On} for direct proofs). More generally, Wang and Dhaene show in \cite{ref26} that a distortion risk measure is coherent if and only if the distortion function is concave (which holds in our case).

\begin{remark}
In the litterature, alternative set of axioms for coherent measure of risks have been considered (in \cite{cono1}, \cite{ref26} or \cite{coherent} for example,  additivity for a particular class of risk (comonotonic risks) is also studied). In this paper we will only focus on the Rockafellar's definition of a coherent measure of risk (see \cite{def}).

Besides, theoretical results have also been shown for coherent measures of risk. These measures can indeed be represented by suprema of linear functionals (see for example \cite{kusu} for the Kusuoka representation or \cite{coherent} for the scenarios set representation). We will not be interested in such representations here.  
\end{remark}

\section{Bregman superquantiles}
In this section the aim is to build a general measure of  risk that satisfies some of the regularity axioms stated in Definition \ref{deco}.  These quantities will be built by using a dissimilarity measure beetween real numbers, the Bregman divergence (see \cite{bre67}).
\subsection{Bregman divergence,  mean and superquantile}
In this section we first  recall the definition of  the Bregman mean of a probability measure $\mu$ (see \cite{bencha89}) and define the measure of risk that we will study. To begin with, we recall the definition of the Bregman divergence that will be used to build the Bregman mean. 
Let $\gamma$ be a strictly convex function, $\overline{\BBr}$-valued on $\BBr$. As usual we set 
$$\dom_{\gamma} :=\{x\in\BBr:\gamma(x)<+\infty\}.$$
For sake of simplicity we assume that $\dom_{\gamma}$ is a non empty open set and that $\gamma$ is a closed proper differentiable function on the interior of $\dom_{\gamma}$ (see \cite{Rockaconvex}). From now we always consider function $\gamma$ satisfying this assumption. The Bregman divergence $d_{\gamma}$ associated to $\gamma$ (see \cite{bre67}) is a function defined on $\dom_{\gamma}\times \dom_{\gamma}$ by
$$d_\gamma(x,x'):= \gamma(x)-\gamma(x')-\gamma'(x')(x-x')\;\;, \; \;  (x,x'\in \dom_{\gamma}).$$
The Bregman divergence is not a distance as it is not symmetric. Nevertheless, as it is non negative and vanishes, if and only if, the two arguments are equal, it quantifies the proximity of points in  $\dom_{\gamma}$. Let us recall some classical examples of such a divergence.
\begin{itemize}
\item Euclidean. $\gamma(x)=x^2$ on $\BBr$, we obviously obtain, for $x,x'\in\BBr$,  
$$d_\gamma(x,x')=(x-x')^2.$$
\item Geometric. $\gamma(x)=x\ln(x)-x+1$ on $\BBr_+^*$ we obtain, for $x,x'\in\BBr_+^*$,  
$$d_\gamma(x,x')=x\ln\frac{x}{x'}+x'-x.$$
\item Harmonic. $\gamma(x)=-\ln(x)+x-1$ on $\BBr_+^*$ we obtain, for $x,x'\in\BBr_+^*$,  
$$d_\gamma(x,x')=-\ln\frac{x}{x'}+\frac{x}{x'}-1.$$
\end{itemize}
Let $\mu$ be a probability measure whose support is included in $\dom_{\gamma}$  and that does not weight the boundary of $\dom_{\gamma}$. Assume further that $\gamma'$ is integrable with respect to $\mu$. Following  \cite{bencha89}, we first define the Bregman mean as the unique point $b$ in the support of $\mu$ satisfying 

\begin{equation}
\int d_\gamma(b, x)\mu(dx)=\min_{m\in\dom_{\gamma}}\int d_\gamma(m, x)\mu(dx).
\label{eq1} 
\end{equation}

In fact, we replace the  $L^{2}$ minimization in  the definition of the mathematical classical expectation by the minimization of the Bregman divergence.
Existence and uniqueness come from the convexity properties of $d_\gamma$ with respect to its first argument. By differentiating it is easy to see that
$$b=\gamma^{'-1}\left[\int\gamma'(x)\mu(dx)\right].$$
Hence, coming back to our three previous examples, we obtain the classical mean  in the first example (Euclidean case), the geometric mean ($\exp\int\ln(x)\mu(dx)$), in the second one and the harmonic mean  ($[\int x^{-1}\mu(dx)]^{-1}$), in the third one. 
Notice that, as the Bregman divergence is not symmetric, 
we have to pay attention to the definition of the Bregman mean. Indeed, we have
$$\int d_\gamma(x, \BBe(X) )\mu(dx)=\min_{m\in\dom_{\gamma}}\int d_\gamma(x,m)\mu(dx).$$ 
We turn now to the definition of our new measure of risk.

\begin{definition}
Let $\alpha\in]0,1[$, the Bregman superquantile $Q_{\alpha}^{d_\gamma}$ is defined by
$$Q_{\alpha}^{d_\gamma}:= \gamma'^{-1}\Big(\mathbb{E}(\gamma'(X)|X \geq  F_{X}^{-1}(\alpha))\Big)=\gamma'^{-1}\left[\mathbb{E}\left(\frac{\gamma'(X)\mathbf{1}_{X \geq  F_{X}^{-1}(\alpha)}}{1-\alpha}\right)\right] $$
where the second egality holds because $F_X$ is continuous. 
In words $Q_{\alpha}^{d_\gamma}$ satisfies (\ref{eq1}) taking for $\mu$ the distribution of $X$ conditionally to 

$X \geq F_X^{-1}(\alpha)$. We now denote $Q_{\alpha}^{d \gamma}$ the Bregman superquantile of the random variable $X$ when there is no ambiguity and $Q_{\alpha}^{d \gamma}(X)$ if we need to distinguish Bregman superquantile of different distributions.

\label{desupbreg}

\end{definition}

\vspace{0.3cm}

For the same reasons as before, the Bregman superquantile is always well-defined as an element of $\bar{\mathbb{R}}$. Moreover, we can already see an advantage of the Bregman superquantile over the classical superquantile. Indeed, some real random variables are not integrable (and so the superquantile is egal to $+ \infty$), but thanks to a choice of a very regular function $\gamma$, the Bregman superquantile can be finite. For example, let us introduce $X$ from the one side Cauchy distribution, that is having density function 

$$f(x)= \frac{2}{\pi(1+x^2)} \mathbf{1}_{x \geq 0}.$$

Since $X$ is not integrable, its classical superquantile is egal to $+ \infty$. Nevertheless, considering the Bregman superquantile associated to the strictly convex fonction 

$\gamma(x)= x\ln(x)-x+1$, we have 

$$\mathbb{E} \left( \gamma'(X) \mathbf{1}_{X \geq F^{-1}_X(\alpha)}\right) < + \infty$$

because the function $x \mapsto \frac{\ln(x)}{1+x^2}$ is integrable on $[0,+\infty[$.

\vspace{0.3cm}

\textbf{Interpretation of the Bregman superquantile :} As a matter of fact we have

\begin{equation}
Q_{\alpha}^{d \gamma}(X)=\gamma'^{-1}\left(Q_{\alpha}(\gamma'(X)\right).
\label{link}
\end{equation}

Indeed, denoting $Z=\gamma'(X)$, as $\gamma'\left(F^{-1}_{X}(\alpha)\right) = F_{Z}^{-1}(\alpha)$, so that

$$\mathbb{E}\left(\gamma'(X) |X > F_{X}^{-1}(\alpha)\right) = \mathbb{E}\left(Z |Z > F_{Z}^{-1}(\alpha)\right).$$

Thus, the Bregman superquantile can be interpreted in the same way that a superquantile under a change of scale. In other words : fix a threshold $\alpha$ and compute the corresponding quantile. Further, change the scale $X \mapsto \gamma'(X)$ and compute the corresponding mean. At last, apply the inverse change of scale to come back to the true space. 

\vspace{0.3cm}

This natural idea has already been used in economy. For example, noticing that the classical Gini index does not satisfy properties that are essential to ensure a reliable modelisation, Satya et al. introduced in \cite{gini} generalized Gini index thanks to a similar change of scale allowing the index to satisfy these properties.

\vspace{0.3cm}

In our case, the main interest of this new measure of risk is also in the change of scale. Indeed, choosing a slowing varying convex function $\gamma$ leads to a more \textit{robust} risk allowing a statistical esimation with better statistical properties (we show for example in Section 3 that empirical estimator for classical superquantile is not always consistent when $X$ has a Pareto distribution, whereas it is always the case with the Bregman superquantile).

\begin{remark}
The Bregman superquantile has a close link with the weighted allocation functional in the capital allocation fields. Indeed, in \cite{eco1}, this quantity is defined as :

$$A_{w}[U, V]:= \frac{E(Uw(V))}{w(V)}$$

where $U$ and $V$ are two real random variables and $w$ is a given map from $\mathbb{R}^+$ to $\mathbb{R}^+$. Choosing $U=X$, $V=\gamma'(X)$ and $w(V)=\mathbf{1}_{V \geq Q_{\alpha}^V}$, we obtain

$$A_{w}[X, \gamma'(X)]=\gamma'\left(Q_{\alpha}^{d\gamma}(X)\right).$$

\end{remark}


\subsection{Coherence of Bregman superquantile}

The following proposition gives some conditions under which the Bregman superquantile is a coherent measure of risk. 

\begin{proposition}
\label{coherence}

Fix $\alpha$ in $]0, 1[$.
$\;$
\begin{itemize}
\item[i)] Any Bregman superquantile always satisfies the properties of constant invariance and monotonicity.
\item[ii)] The Bregman superquantile associated to the function $\gamma$ is homogeneous, if and only if, 

$\gamma''(x) = \beta x^{\delta}$ for some real numbers $\beta>0$ and $\delta$ (as $\gamma$ is convex, if the support of $\gamma$ is strictly included in $\mathbb{R}^{+}_*:=]0, + \infty[$ there is no condition on $\delta$ but if not, $\delta$ is an even number).
\item[iii)] If $\gamma'$ is concave and sub-additive, then subadditivity and closeness axioms both hold. 
\end{itemize}
\end{proposition}

The proof of this proposition, like all the others, is differed to Section 5.

To conclude, under some regularity assumptions on $\gamma$, the Bregman superquantile is a coherent measure of risk. Let us take some examples.

\subsubsection{Examples and counter-examples}

\begin{itemize}
\item \textbf{Example 1 :}  $x \mapsto x^{2}$ satisfies all the hypothesis but it is already known that the classical superquantile is subaddtive.
\item \textbf{Example 2 :} The Bregman geometric and harmonic functions satisfies the assumptions i) and ii). Moreover, their derivatives are, respectively, $x \mapsto \gamma'(x) = \ln(x)$ and $x \mapsto \gamma'(x)= \frac{x-1}{x}$  which are concave but subadditive only on $[1, + \infty[$. Then the harmonic and geometric functions satisfy iii) not for all pairs of random variables but only for pairs $(X, X')$ such that, denoting $Z:=X+X'$ we have

 $$\min \left(q_{\alpha}^{X}(\alpha), q_{\alpha}^{X'}(\alpha), q_{\alpha}^{Z}(\alpha) \right) >1.$$ 
 
\item \textbf{Example 3 :} A classical strictly convex function in economy (for example for the computations of the extended Gini index in \cite{gini}) is the function $\gamma(x)=x^{\alpha}$ with $\alpha >1$ when considering random variables which support is included in $R^+:=[0, + \infty[$. This convex function satisfies axiom ii) of our proposition, so the associated Bregman superquantile is homogeneous. Moreover, $\gamma'(x)=\alpha x^{\alpha-1}$ is concave if, and only if, $\alpha < 2$. In this case, it is subadditive on $[0, + \infty[ $ as a concave function such that $f(0) \geq 0$. Finally, the Bregman superquantile associated to the function $\gamma(x)=x^{\alpha}$, $1< \alpha <2$ is a coherent measure of risk when considering non-negativ random variables. 

\item \textbf{Counter-example 4 :} The subadditivity is not true in the general case. Indeed, let $\gamma(x)=\exp(x)$ and assume that $X \sim \mathcal{U}\left([0, 1]\right)$.

$$\mathbb{E}\left(\gamma'(X) \mathbf{1}_{X \geq F_{X}^{-1}( \alpha)}\right) = \int_{\alpha}^{1} \exp( x) dx = e - \exp(\alpha).$$

Then, denoting $\mathcal{R}(V)=Q_{\alpha}^{d\gamma}(V)$ for $V$ a random variable, we get

$$ \mathcal{R}( X) = \ln\left( \frac{e - \exp(\alpha)}{1-\alpha}\right).$$

Moreover, 

$$\mathcal{R}(\lambda X)= \ln \left(\int_{\alpha}^{1} \exp( \lambda x) dx\right)=\ln \left( \frac{\exp(\lambda) - \exp((\alpha) \lambda)}{\lambda(1-\alpha)} \right).$$

For $\alpha=0.95$ and $\lambda =2$, we obtain 

$$\mathcal{R}(2X) - 2\mathcal{R}(X) = \mathcal{R}(X+X) - \left(\mathcal{R}(X)+\mathcal{R}(X) \right) =  0.000107 >0,$$
and subadditivity fails. 

We can also notice that for $\lambda=4$

$$\frac{\mathcal{R}(4 X)}{4 \mathcal{R}(X)} = 1, 000321,$$

and the positive homogeneity is not true which is coherent with the Proposition \ref{coherence} since the derivative of $\gamma$ does not fulfill the assumption. 

\end{itemize}

\subsection{Remarks toward other natural properties}

We study the Bregman superquantile as a measure of risk. It is then natural to wonder if other classical properties of measure of risk are satisfied by this new quantity. Let us make some remarks.

\begin{itemize}
\item[1)] First, we can study the continuity of the Bregman superquantile. A condition for classical superquantile to be continuous, that is to have 

$$X_n \underset{a.s}{\longrightarrow} X \, \, \Rightarrow Q_{\alpha}(X_n) \rightarrow Q_{\alpha}(X_n)$$

is that the sequence $(X_n)$ is equi-integrable. Then the continuity of the Bregman superquantile is true when the sequence $\left(\gamma'(X_n)\right)_n$ is equi-integrable. 

We thus put forward an other advantage of the Bregman superquantile over the classical superquantile, because the transformation throught $\gamma$ can regularized the sequence and make it equi-integrable. Indeed, let us consider a sample $(X_n)$ of independent copies of $X$ where $X$ has the truncated Cauchy distribution. We have already seen that $X$ it not integrable. Then the sequence $(X_n)$ is not bounded in $L^1$ and so not equi-integrable. But, with the function $\gamma(x)=x\ln(x)-x+1$, the random variable $\gamma'(X)$ is integrable. Then the independent sample $\left(\gamma'(X_n) \right)_n$ is equi-integrable.

\item[2)] The relation exhibited in Equation (\ref{link}) allows us to deduce some properties for the Bregman superquantile from the classical superquantile. For example, Gneiting et al. show in \cite{sqeli} that the classical superquantile is not elicitable (notion introduced in \cite{eli}). Then, an easy proof by reduction show that, the Bregman superquantile is not elicitable. Likewise, Cont et al. show in \cite{robu} that the classical superquantile is not robust (in particular because it is not subadditive). A direct consequence (because the fonction $\gamma'$ is continuous and the Levy distance is the distance associated to the weak convergence) is that the Bregman superquantile is not robust either. The Bregman superquantile is another example of what Cont et al. calls the conflict between subadditivity and robustness.

\end{itemize}

\section{Estimation of the Bregman superquantile}

In this section the aim is to estimate the Bregman superquantile from a sample. We introduce a Monte Carlo estimator and study its asymptotics properties. Under regularity assumptions on the functions $\gamma$ and $F_X^{-1}$, the Bregman superquantile is consistent and asymptotically Gaussian. All along this section, we consider a function $\gamma$ satisfying our usual properties and a real-valued random variable $X$ such that $\gamma'(X)\textbf{1}_{X \geq 0}$ is integrable.

\subsection{Monte Carlo estimator}
Assume that we have at hand $(X_{1}, \dots, X_{n})$ an i.i.d sample with same distribution as $X$. If we wish to estimate $Q^{d_{\gamma}}_{\alpha}$, we may use the following empirical estimator :

\begin{equation}
\begin{aligned}
\label{estimateurSQB}
\widehat{Q_{\alpha}^{d_{\gamma}}}&= \gamma^{'-1}\left[ \frac{1}{1-\alpha} \left( \frac{1}{n} \sum_{i= \lfloor n \alpha \rfloor +1} ^{n} \gamma'(X_{(i)})\right) \right],
\end{aligned}
\end{equation}

 where $X_{(1)} \leq X_{(2)} \leq \dots \leq X_{(n)}$ is the re-ordered sample built with $(X_{1}, \dots, X_{n})$.

\subsection{Asymptotics}

We give a theorem which gives the asymptotic behaviour of the Bregman superquantile. The following assumptions are usefull for our next theorem.

\begin{itemize}

%
\item [\textbf{H1)}] The function $\gamma' \circ F_X^{-1}$ is of class $\mathcal{C}^1$ on $]0, 1[$ and its derivative that we denote by $l_{\gamma}$ satisfies $l_{\gamma}(t)=O \left( (1-t)^{-2+\epsilon_{l_{\gamma}}} \right)$ for an $\epsilon_{l_{\gamma}}>0$ when $t$ goes to $1^-$.

\item[\textbf{H2)}] The function $\gamma' \circ F_X^{-1}$ is of class $\mathcal{C}^2$ on $]0, 1[$ and its second derivative that we denote by $L_{\gamma}$ satisfies  $L_{\gamma}(t)=O \left( (1-t)^{-\frac{5}{2}+\epsilon_{L_{\gamma}}} \right)$ for an $\epsilon_{L_{\gamma}}>0$ when $t$ goes to $1^-$.

\begin{remark}
\label{petito}
Assumption \textbf{H1} implies that $X$ is absolutely continuous of density $f_X$ which is continuous and positive and that $\gamma$ is of class $\mathcal{C}^2$. It also implies that $l_{\gamma}=o \left( (1-t)^{-2} \right)$ around 1. Assumption \textbf{H2} implies that $f_X$ is also of class $\mathcal{C}^1$ and $\gamma$ of class $\mathcal{C}^3$. It also implies that $L_{\gamma}(t) = o\left((1-t)^{- \frac{5}{2}} \right)$.
\end{remark}

\end{itemize}

\begin{theorem}
\label{bregasympto}

Let $\frac{1}{2} < \alpha < 1$ and $X$ be a real-valued random variable. Let $(X_1, \dots X_n)$ be an independent sample with the same distribution as $X$.

\begin{itemize}

\item [i)] Under assumption \textbf{H1}, the estimator $\widehat{Q_{\alpha}^{d_{\gamma}}}$ is consistent in probability :

$$ \widehat{Q_{\alpha}^{d_{\gamma}}} \overset{\mathbb{P}}{\underset{n \to + \infty}{\longrightarrow}} Q_{\alpha}^{d\gamma}.$$

\item[ii)] Under assumption \textbf{H2}, the estimator $\widehat{Q_{\alpha}^{d_{\gamma}}}$ is asymptotically normal  :

$$\sqrt{n} \left( \widehat{Q_{\alpha}^{d_{\gamma}}} - Q_{\alpha}^{d \gamma} \right) \underset{n \to \infty}{\Longrightarrow} \mathcal{N}\left(0, \frac{\sigma^{2}_{\gamma}}{ \left(\gamma'' \left(Q_{\alpha}^{d \gamma}(X) \right)\right)^2 (1-\alpha)^2}\right), $$

 where $$\sigma^2_{\gamma} := \int_{\alpha}^1 \int_{\alpha}^1 \frac{(\min(x, y) - xy)}{f_Z(F_Z^{-1}(x))f_Z(F_Z^{-1}(y))} dx dy.$$

and $Z:=\gamma'(X)$.

\end{itemize}

\end{theorem}

\begin{remark}
Easy calculations show that we have the following equalities $$l_{\gamma}:= \frac{\gamma'' \circ F_X^{-1}}{f_X \circ F_X^{-1}},$$ $$L_{\gamma}:= \frac{\left( \gamma''' \circ F_X^{-1}) \times \left( f_X \circ F_X^{-1} \right)- \left( f_{X}' \circ F_X^{-1} \right) \times (\gamma'' \circ F_X^{-1} \right)}{\left(f_X \circ F_X^{-1}\right)^3}$$ and $$f_Z=\frac{f_X \circ \gamma'^{-1}}{\gamma'' \circ \gamma'^{-1}}.$$
\end{remark}

\begin{remark}
The second part of the theorem shows the asymptotic normality of the Bregman superquantile empirical estimator. We can then use the Slutsky's lemma to find confidence intervals. Indeed, since our estimator (\ref{estimateurSQB}) is consistent, we also have 

$$\frac{\sqrt{n}}{\left( \gamma'' \circ \widehat{Q_{\alpha}^{d \gamma}}\right)} \left( \widehat{Q_{\alpha}^{d \gamma}}- Q_{\alpha}^{d \gamma}(X) \right) \Longrightarrow \mathcal{N}\left(0, \frac{\sigma^{2}_{\gamma}}{(1-\alpha)}\right).$$
\end{remark}

To prove Theorem \ref{bregasympto} we use the following results on the asymptotic properties of the superquantile (which is equivalent to deal with the Bregman superquantile when the function $\gamma$ equals to identity), thanks to the link established in Equation (\ref{link}). Asymptotic behaviour of plug-in estimator for general distortion risk measure has already been studied (see for example \cite{lui} for strong consistency and \cite{ref6} for a central limit theorem). Nevertheless, for sake of completeness we propose in this paper a simpler and self-contained proof of these results for the particular case of the superquantile. Further, our proof also allows to exhibit the explicit asymptotic variance which makes the study of Bregman superquantile estimator easier. Indeed, we can then apply these results to the sample $(Z_1, \dots, Z_n)$ where $Z_i:=\gamma'(X_i)$. We conclude by applying the continuous mapping theorem with $\gamma'^{-1}$ which is continuous thanks to assumption \textbf{H1} for the consistency, and by applying the delta-method (see for example \cite{Vander}), with the function $\gamma'^{-1}$ which is differentiable thanks to assumption \textbf{H2} and of positive derivative since $\gamma$ is strictly convex. 

\vspace{0.3cm}

Then, let us consider a real valued random variable $X$ such that $X\mathbf{1}_{X \geq 0}$ is integrable. With the previous notations, if we wish to estimate the classical superquantile $Q_{\alpha}$ we may use the following empirical estimator

\begin{equation}
\label{estimateurSQ}
\begin{aligned}
Q_{\alpha}:= \frac{1}{(1-\alpha)n} \displaystyle\sum_{i= \lfloor n \alpha \rfloor}^n X_{(i)}
\end{aligned}
\end{equation}

\vspace{0.3cm}

For the next proposition we need the two following assumptions.

\begin{itemize}

\item [\textbf{H3)}] The quantile function $F_X^{-1}$ is of class $\mathcal{C}^1$ on $]0, 1[$ and its derivative, denoting $l$, satisfies $l=O((1-t)^{-2+ \epsilon_l})$ for an $\epsilon_l>0$ when $t$ goes to $1^-$.

\item[\textbf{H4)}] The quantile function is $\mathcal{C}^2$ on $]0, 1[$ and its second derivative that we denote $L$ satisfies $L=O \left( (1-t)^{-\frac{5}{2}+ \epsilon_L} \right)$ for an $\epsilon_L >0$ when $t$ goes to $1^-$. 
%
%

\end{itemize}

\begin{proposition}
\label{superasympto}

Let $\frac{1}{2} < \alpha < 1$, and $X$ be a real-valued random variable. Let $(X_1, \dots X_n)$ be an independent sample with the same distribution as $X$.

\begin{itemize}

\item [i)] Under \textbf{H3}, the empirical estimator (\ref{estimateurSQ}) is consistent in probability

$$ \widehat{Q_{\alpha}} \overset{\mathbb{P}}{\underset{n \to + \infty}{\longrightarrow}} Q_{\alpha}.$$

\item[ii)] Under \textbf{H4}, the  empirical estimator (\ref{estimateurSQ}) is asymptotically Gaussian

$$\sqrt{n} \left( \frac{1}{n(1-\alpha)} \displaystyle\sum_{i= \lfloor n \alpha \rfloor +1}^{n} X_{(i)} - Q_{\alpha}\right) \Longrightarrow \mathcal{N}\left(0, \frac{\sigma^{2}}{(1- \alpha)^2}\right),$$

where $$\sigma^2 := \int_{\alpha}^1 \int_{\alpha}^1 \frac{(\min(x, y) - xy)}{f(F^{-1}(x))f(F^{-1}(y))}.$$

\end{itemize}

\begin{remark}
\label{appelle}
Remark \ref{petito} still holds for $l$ instead of $l_{\gamma}$ and $L$ instead of $L_{\gamma}$. 
\end{remark}

\end{proposition}

\subsection{Examples of asymptotic behaviors for the classical superquantile}\label{sec:examples}

Our assumptions are easy to check in practice. Let us show some examples of the asymptotic behaviour of the estimateur of the superquantile (\ref{estimateurSQ}) by using the exponential distribution of parameter 1 and the Pareto distribution.

\subsubsection{Exponential distribution}

In this case, we have on $\mathbb{R}^{+}_{*} \, \,  f(t)= \exp(-t)$, 

$F(t)=1-  \exp(-x)$. Then $F^{-1}(t)= - \ln(1-t)$. 

\begin{itemize}
\item Consistency :

$l(t)= (1-t)^{-1} = O\left((1-t)^{-2+ \frac{1}{2}} \right)$  (when $t \mapsto 1^-$) so that the estimator (\ref{estimateurSQ}) is consistent.

\item Asymptotic normality :

$L(t)= (1-t)^{-2} = O \left((1-t)^{-\frac{5}{2}+ \frac{1}{3}} \right)$ (when $t \mapsto 1^-$). So the estimator (\ref{estimateurSQ}) is asymptotically Gaussian.
\end{itemize}

\subsubsection{Pareto distribution}

Here, we consider the Pareto distribution of parameter $a>0$ : on $\mathbb{R}^{+}_{*}$, 

$F(t)= 1 - x^{-a}$, $f(t)= a x^{-a-1}$, and $F^{-1}(t) = (1-t)^{\frac{-1}{a}}$. 

\begin{itemize}

\item Consistency :

$l(t)=(a (1-t)^{-1 - \frac{1}{a}})$ thus,  $l(t)=O\left( (1-t)^{-2+ \epsilon_l} \right)$ (when $t \mapsto 1^-$) as soon as $a>1$ (for exemple for $\epsilon_l= \frac{1- \frac{1}{a}}{2}$). The consistency for estimator (\ref{estimateurSQ}) is true.

\item Asymptotic normality :

$L(t) = C(a) (1-x)^{-\frac{1}{a}-2}$ thus, as soon as $a>2$, $L(t)=  O \left( (1-t)^{- \frac{5}{2}+ \epsilon_L} \right)$ (for example $\epsilon_L= \frac{\frac{1}{2}- \frac{1}{a}}{2}$), when $t \mapsto 1^-$. The estimator (\ref{estimateurSQ}) is asymptotically gaussian if and only if $a>2$.

\end{itemize}

\subsection{Examples of asymptotic behaviour of the Bregman superquantile}

Let us now study the same examples in the case of the Bregman superquantile and its empirical estimator (\ref{estimateurSQB}). For the exponential distribution, the conclusion is the same as in the previous case of classical superquantile. However, for the Pareto distribution, we can find a function $\gamma$ such that the estimator of the Bregman superquantile is asymptotically normal without any condition on the exponent $a$ involved in the distribution. So, the Bregman superquantile is more interesting than the superquantile in this example.

\subsubsection{Exponential distribution}

Let us show the example of the exponential distribution (for $X$) and the harmonic Bregman function (for $\gamma$). We have $\gamma'(x) = (x-1)x^{-1}$ and $F^{-1}_X(t)=- \ln(1-t)$. So that, denoting $Z= \gamma'(X)$ as in Theorem \ref{bregasympto}, $$F_Z^{-1}(t) = 1 + \frac{1}{\ln(1-t)},$$

\begin{itemize}
\item Consistency. In this case, we have, $$l_{\gamma}(t)= \frac{1}{(1-t)(\ln(1-t))^2}. $$ 
So,  $l_{\gamma}$ is a $O \left( (1-t)^{-2+ \frac{1}{2}} \right)$ (when $t \mapsto 1^-$). The estimator (\ref{estimateurSQB}) is consistent.

\item Asymptotical normality.

$$L_{\gamma}(t) = \frac{(\ln(1-t))^2 + 2 \ln(1-t)}{(1-t)^2(\ln(1-t))^4},$$
Then  $L_{\gamma}$  is $O \left( (1-t)^{-\frac{5}{2}+ \frac{1}{3}}\right)$ (when $t \mapsto 1^-$). The estimator (\ref{estimateurSQB}) is asymptotically Gaussian.
\end{itemize}

\subsubsection{Pareto distribution}

Let us now study the case of the Pareto distribution with the geometric Bregman function. We have $F^{-1}_X(t)=(1-t)^{\frac{-1}{a}}$ and $\gamma'(t)=\ln(t)$. Then $$F_Z^{-1}(t)=- \frac{1}{a} \ln(1-t).$$

\begin{itemize}

\item Consistency.

$$l_{\gamma}(t)= \frac{1}{a} \frac{1}{1-t} = O \left(\frac{1}{(1-t)^{2- \frac{1}{2}}} \right),$$

and the monotonicity is true. The estimator (\ref{estimateurSQB}) is consistent.

\item Asymptotic normality.

$$L_{\gamma}(t) = \frac{1}{a} \frac{1}{(1-t)^2}= O \left( (1-t)^{-\frac{5}{2}+ \frac{1}{3}}\right).$$

The estimator (\ref{estimateurSQB}) is consistent and asymptotically gaussian for every $a > 0$.

\end{itemize}

See next part for an illustration of these results by simulations and a summary.

\section{Numerical simulations}

In our numerical tests we simulate values from a known theoretical distribution and computing the $0.95$-quantiles and superquantiles.
For each estimated quantity, the reference value is given via a $10^6$-size random sample and a convergence study is performed from a $1
000$-size sample to a $10^5$ size sample (with a step of $500$).
In order to annihilate the effect of randomness, $50$ repetitions of each numerical experiment are made. Then, we compute
\begin{itemize}
\item The mean value of the $50$ estimations to be compared to the reference value,
\item The standard deviation of the $50$ estimations. It allows to compute an experimental $95\%$-confidence interval (CI) to be compared to the theoretical $95\%$-CI (given by the central limit theorem).
\end{itemize}
Each is composed of four plots of convergence for the following quantities: quantile (up left), classical superquantile (up right), geometrical superquantile (bottom left) and harmonic superquantile (bottom right).
Each superquantile convergence plot is composed of the following curves: Reference value (dotted black line), mean estimated values (red circles), theoretical $95\%$-CI (dashed black line) and experimental $95\%$-CI (solid blue line).
 
Figure \ref{fig:SQexp1} gives the results for an exponential distribution of parameter $\lambda=1$.
As predicted by the theory (see Section \ref{sec:examples}), for the three different superquantiles, the consistency is verified while the experimental CI perfectly fits the theoretical CI (given by  the central limit theorem).

\begin{figure}[!ht]
$$\includegraphics[height=8cm]{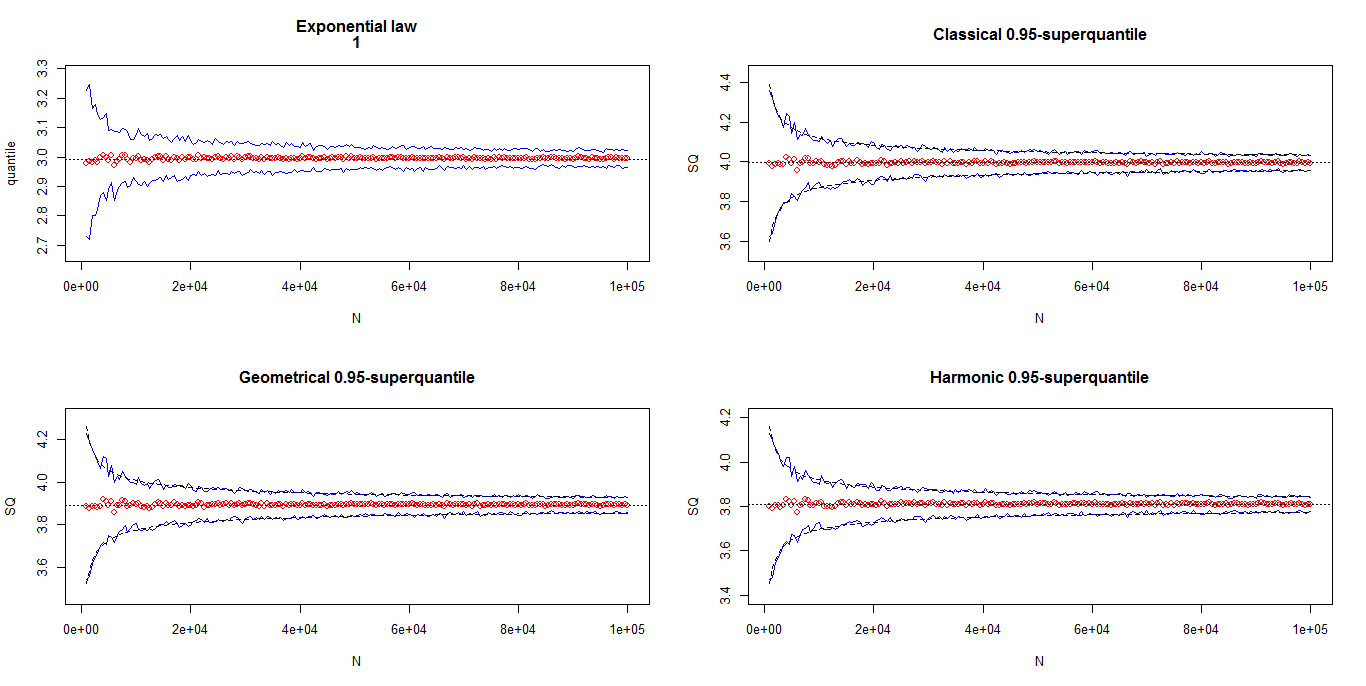}$$
\caption{Numerical convergence test for the exponential distribution.}\label{fig:SQexp1}
\end{figure}

We then test the Pareto distribution (see Section \ref{sec:examples}) with three different shape parameters: $a=0.5$, $a=1.5$ and $a=2.5$.
Figures \ref{fig:SQpareto0.5} ($a=0.5$), \ref{fig:SQpareto1.5} ($a=1.5$) and \ref{fig:SQpareto2.5}  ($a=2.5$) give the convergence results.
For the geometrical and harmonic superquantiles, as predicted by the theory (see Section \ref{sec:examples}), the consistency of the Monte Carlo estimation is verified while the experimental CI perfectly fits the theoretical CI (asymptotic normality).
For the classical superquantile, we distinguish three different behaviors: 
\begin{itemize}
\item No consistency for $a=0.5$ (Figure \ref{fig:SQpareto0.5}) (theory predicts consistency only if $a>1$),
\item Consistency but no asymptotic normality for $a=1.5$ (Figure \ref{fig:SQpareto1.5}) (theory predicts asymptotic normality only if $a>2$),
\item Consistency and asymptotic normality for $a=2.5$ (Figure \ref{fig:SQpareto2.5}),
\end{itemize}

\begin{figure}[!ht]
$$\includegraphics[height=8cm]{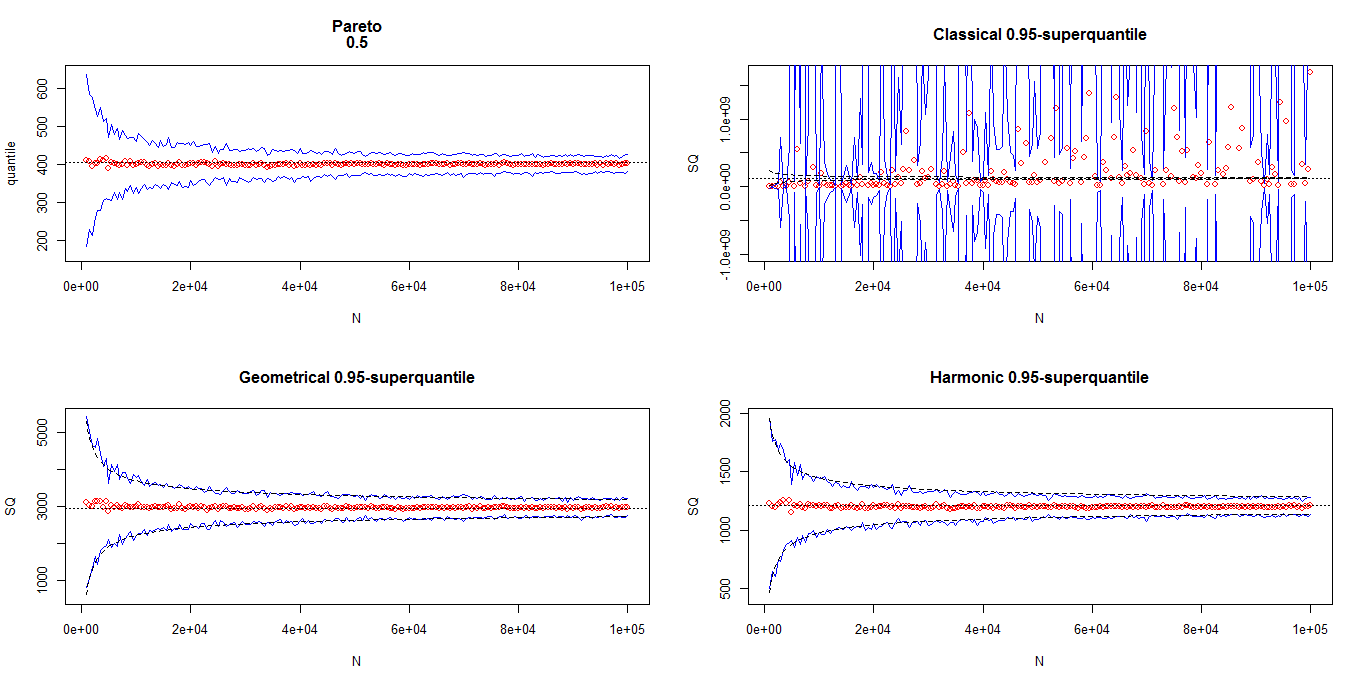}$$
\caption{Numerical convergence test for the Pareto distribution ($a=0.5$).}\label{fig:SQpareto0.5}
\end{figure}

\begin{figure}[!ht]
$$\includegraphics[height=8cm]{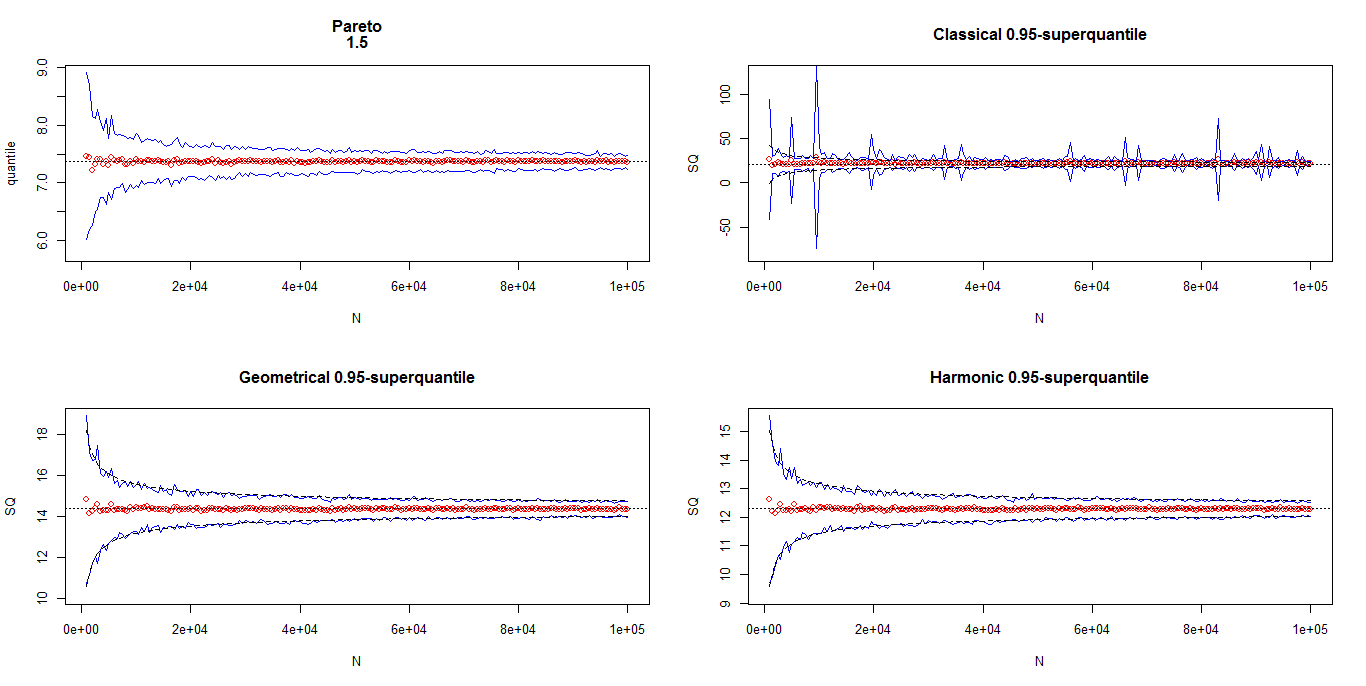}$$
\caption{Numerical convergence test for the Pareto distribution ($a=1.5$).}\label{fig:SQpareto1.5}
\end{figure}

\begin{figure}[!ht]
$$\includegraphics[height=7cm]{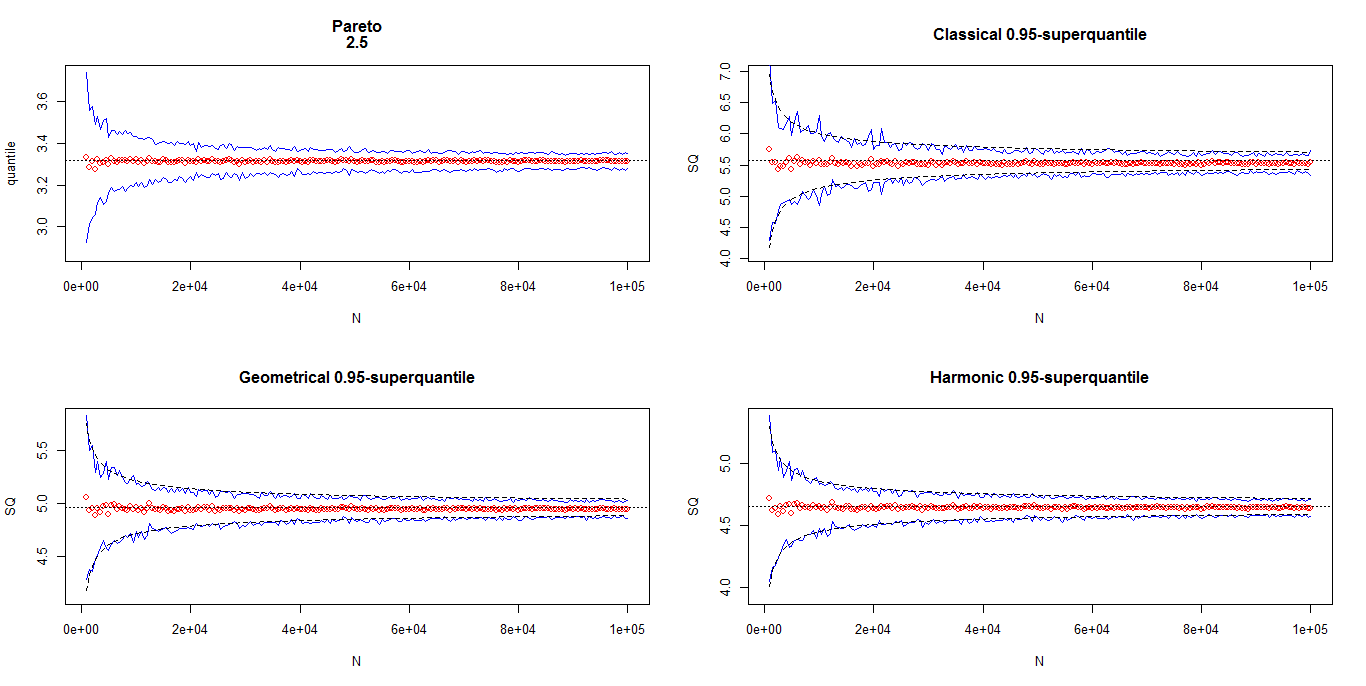}$$
\caption{Numerical convergence test for the Pareto distribution ($a=2.5$).}\label{fig:SQpareto2.5}
\end{figure}

To sum up the two previous parts, the plug in estimators (\ref{estimateurSQ}) and (\ref{estimateurSQB}) seem to have the same asymptotic behaviour when considering distribution with not too heavy tail-distribution, like the exponential distribution. Nevertheless, the estimator of the Bregman superquantile (\ref{estimateurSQB}) has better statistical properties when we deal with heavy tail-distribution. A typical example is the Pareto distribution. Indeed, with Parato distribution of parameter $a>2$, the tail is not too heavy and the both estimator have good asymptotic properties. This not the case any more when we choose parameter $a<2$. When $1<a<2$, the estimator of the classical superquantile (\ref{estimateurSQ}) is no more asymptotically gaussian and when $0<a<1$ it is even not consistent, whereas the estimator for Bregman superquantile (\ref{estimateurSQB}) keeps good asymptotic properties in each case. 

\section{Applications to a nuclear safety exercise}

GASCON is a software developed by CEA (French Atomic Energy Commission) to study the potential chronological
atmospheric releases and dosimetric impact for nuclear facilities safety assessment \cite{ioovan06}.
It evaluates, from a fictitious radioactive release,
the doses received by a population exposed to the cloud of radionuclides and through the food chains. It takes into account the interactions that exist between humans, plants and animals, the different pathways of transfer (wind, rain, \ldots), the distance between emission and observation, and the time from emission.

As GASCON is relatively costly in computational time, in \cite{ioovan06}, the authors have built metamodels (of polynomial form) of GASCON outputs in order to perform uncertainty and sensitivity analysis.
As in \cite{jaclav06}, we focus on one output of GASCON, the annual effective dose in $^{129}$I received in one year by an adult who lives in the neighborhood of a particular nuclear facility. 
Instead of the GASCON software, we will use here the metamodel of this output which depends on $10$ input variables, each one modelled by a log-uniform random variable (bounds are defined in \cite{ioovan06}).
The range of the model output stands on several decades ($10^{-14}$ to $10^{-11}$ Sv/year) as shown by Figure \ref{fig:gascon-hist} which represents the histogram (in logarithmic scale) of $10^6$ simulated values.

\begin{figure}[!ht]
$$\includegraphics[height=7cm]{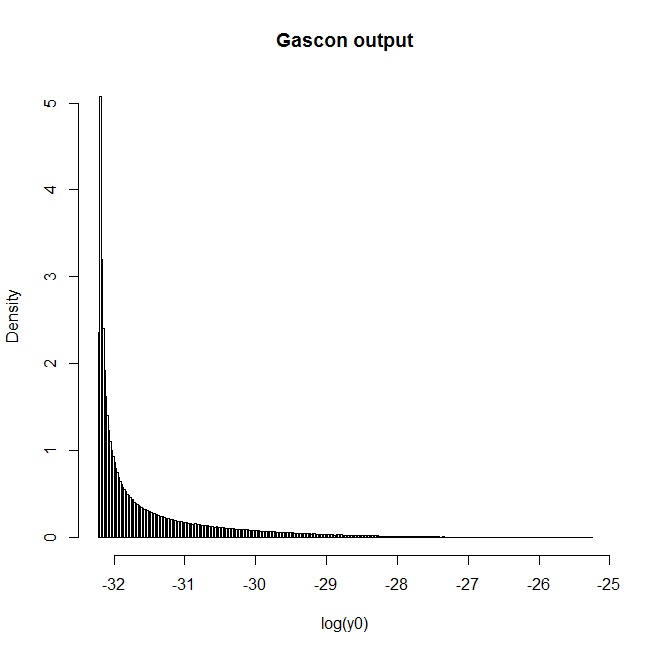}$$
\caption{Distribution of the GASCON output variable.}\label{fig:gascon-hist}
\end{figure}

For this kind of numerical simulation exercises, we can be typically interested by safety criteria as $95\%$-quantile and its associated superquantiles.
The idea is to compare these values to regulatory limits or to results coming from other scenarios or from other tools.
In practice, the number of simulations performed with the GASCON model is several hundreds.
Table \ref{tab:gascon-SQ} gives the estimated values of the quantile and superquantiles for $1000$ metamodel simulations.
Figure \ref{fig:gascon-SQ} shows the relative errors (computed by averaging $1000$ different estimations) which are made when estimating the superquantiles using $3$ different Bregman divergences and with different sampling sizes.
We observe that geometrical and harmonic superquantiles are clearly more precise than the classical one. Using such measures is therefore more relevant when performing comparisons.

\begin{table}[!ht]
\caption{Estimated values of $95\%$-quantile and $95\%$-superquantiles for $1000$ simulations.}\label{tab:gascon-SQ}
\begin{tabular}{c|c|c|c}
Quantile & Classical & Geometrical & Harmonic \\
 & superquantile & superquantile & superquantile\\
\hline
$1.304\times 10^{-13}$ & $4.769\times 10^{-13}$ & $3.316\times 10^{-13}$ & $2.637\times 10^{-13}$\\
\end{tabular}
\end{table}

\begin{figure}[!ht]
$$\includegraphics[height=7cm]{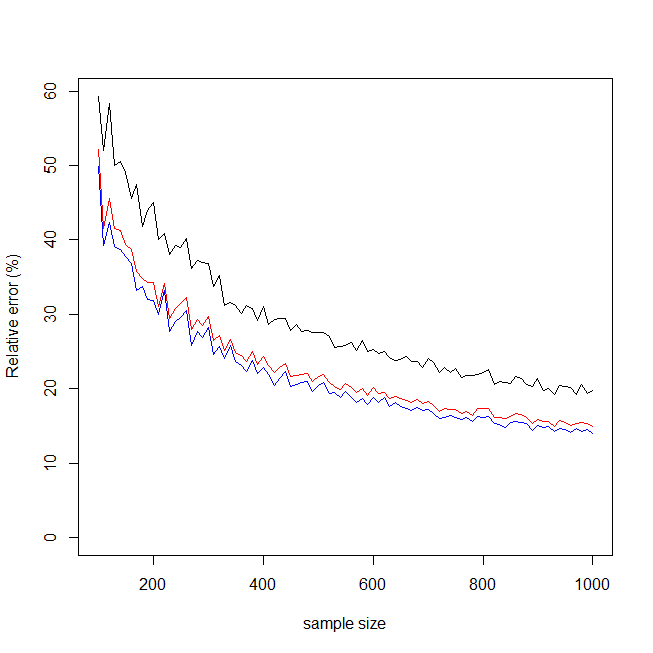}$$
\caption{Evolution of the relative errors (mean square error divided by the reference value) on the estimated superquantiles  in function of the sample size.
In black: classical superquantile; in red: geometrical superquantile; in blue: harmonic superquantile.
The reference value has been calculated with $10^7$ simulations.}\label{fig:gascon-SQ}
\end{figure}

\section{Proofs}

\subsection{Proof of the Proposition \ref{coherence} : coherence of the Bregman superquantile}

\begin{proof}

\textbf{Proof of i) :}

First we obviously have $ Q_{\alpha}^{d\gamma}(C) = \gamma'^{-1}( \gamma'(C))=C$. 

The monotonicity property is well-known for the superquantile (see for example \cite{archiv}, \cite{croissance} or \cite{On}). Then, \eqref{link} and the monotonicity of $\gamma'^{-1}$ and $\gamma'$ (since $\gamma$ is strictly convex, its derivative $\gamma'$ is strictly non-decreasing and so is its inverse function $\gamma'^{-1}$) allow us to conclude. 

%
%
%
%
%

\vspace{0.3cm}

\textbf{Proof of ii) :}

Let us reformulate the problem. For every (measurable) function $f$ and for every random variable $X$, we denote $$\Ea[f(X)] = \BBe[f(X)|X\geq F_X^{-1}(\alpha)].$$ Let $X$ and $X'$ be two real-valued random variable.
The Bregman superquantile of $X$ associated to $\gamma$ is 
\[\gamma'^{-1}\left(\Ea[\gamma'( X)]\right)\,.\]
According to Definition \ref{deco}, the Bregman superquantile associated $\gamma$ is \emph{homogeneous} if, for every random variable $X$ and every $\lambda>0$,

\begin{equation}
\label{preuveaurelien}
\begin{aligned}
\gamma'^{-1}\left(\Ea[\gamma'(\lambda X)]\right) &= \lambda \gamma'^{-1}\left(\Ea[\gamma'(X)]\right)\\
\end{aligned}
\end{equation}

As $\gamma'$ and $x\mapsto (\gamma'(x)-\gamma'(1))/\gamma''(1)$ yield the same superquantiles, one may assume without loss of generality that $\gamma'(1)=0$ and that $\gamma''(1) = 1$

First, it is easy to check that the condition given is sufficient.  For simplicity, we write $\phi = \gamma'$. Let us show that (\ref{preuveaurelien}) holds for each possible form of $\Phi$ given in the proposition. 
If $\phi(x) = (x^\beta-1)/\beta$, then $\phi^{-1}(y) =  (1+\beta y)^{1/\beta}$ and 
\begin{align*}
 \phi^{-1}\left(\Ea[\phi(\lambda X)]\right) &=\left(1+\beta \Ea\left[\frac{(\lambda X)^\beta -1}{\beta}\right]\right)^{\frac{1}{\beta}}\\
 &= \lambda \left(\Ea\left[X^\beta\right] \right)^{1/\beta}\\
 &= \lambda\left(1+\beta \Ea\left[\frac{X^\beta -1}{\beta}\right]\right)^{\frac{1}{\beta}}\\
 &= \lambda  \phi^{-1}\left(\Ea[\phi( X)]\right) \,.
\end{align*}

If $\phi(x)= \ln(x)$, then $\phi^{-1}(y)=\exp(y)$ and 

\begin{align*}
 \phi^{-1}\left(\Ea[\phi(\lambda X)]\right) &=\exp\left(\Ea\left( \ln( \lambda X) \right) \right)\\
 &= \exp\left(\Ea\left( \ln( \lambda) \right) + \Ea\left( \ln( X) \right) \right)\\
 &= \lambda \exp \left( \Ea \left( \ln(X) \right) \right)\\
 &= \lambda  \phi^{-1}\left(\Ea[\phi( X)]\right) \,.
\end{align*}

Let us now show that the condition on $\Phi$ is necessary for (\ref{preuveaurelien}) holds. Let $y>0$. Let $\gamma$ be a strictly convex function such that its associated Bregman superquantile is positively homogeneous : for avery random variable $X$ and for every $\lambda >0$, \ref{preuveaurelien} holds. In particular, let $Y$ be a random variable with distribution $\P$ such that, denoting $a = y\wedge 1$, $\P(du) = \alpha a ^{-1}1_{[0,a]}(u)du + (1-\alpha)p\delta_y + (1-\alpha)(1-p) \delta_1$. Its quantile of order $\alpha$ is $F_Y^{-1}(\alpha) = a$. The conditional distribution of $Y$ given $Y\geq F_Y^{-1}(\alpha)$ is $(1-p)\delta_1 + p\delta_y$, and $\Ea[\phi(Y)] = (1-p)\phi(1) + p \phi(y)$. 
The positive homogeneity property (\ref{preuveaurelien}) and the assumption $\phi(1)=0$ imply that
\[\phi^{-1}\left((1-p) \phi(\lambda) + p\phi(\lambda y)\right) = \lambda \phi^{-1}\left( p\phi(y)\right)\,.\]
By assumption, the expressions on both sides are smooth in $p$ and $y$. Taking the derivative in $p$ at $p=0$ yields \[\frac{\phi(\lambda y) -\phi(\lambda)}{\phi'(\lambda)} = \lambda \frac{\phi(y)}{\phi'(1)}\,,\] 
and hence, as $\phi'(1)=1$,
\[\phi(\lambda y) -\phi(\lambda) = \lambda \phi'(\lambda) \phi(y)\,.\]
By differentiating with respect to $y$, one gets
\begin{equation}\phi'(\lambda y) = \phi'(\lambda)\phi'(y)\,.\label{eq:phi}\end{equation}

Let $\psi$ be defined on $\BBr$ by $\psi(z) = \ln \left(\phi'(\exp( z))\right)$. 
One readily checks that Equation~\eqref{eq:phi} yields
\[\psi\left(\ln(y) + \ln(\lambda)\right) = \psi\left(\ln(y)\right) + \psi\left(\ln(\lambda)\right)\,. \]
This equation holds for every $y,\lambda>0$. This is well known to imply the linearity of $\psi$: there exists a real number $\beta$ (which will be positive since $\gamma$ is convex) such that for all $z\in\BBr$,
\[\ln\left(\phi'(\exp(z))\right) = \psi(z) = \beta z\,.\]
Thus, $\phi'(\exp(z)) = \exp(\beta z)$, that is 
$\phi'(y) = y^{\beta}$ for all $y>0$. 
For $\beta = -1$, one obtains $\phi(y) = \ln(y)$. Otherwise, taking into account the constraint $\phi(1)=0$, this yields
\[\phi(y) = \frac{y^{1+\beta}-1}{1+\beta}\,.\]

\vspace{0.3cm}

\textbf{Proof of iii) : }

Let $X$ and $X'$ be two real-valued random variable. Still denoting, $\mathcal{R}(x)=Q_{\alpha}^{d \gamma}(X)$, we want to show that

$$\mathcal{R}(X+X') \leq \mathcal{R}(X)+ \mathcal{R}(X').$$

Since $\gamma$ is convex,  $\gamma'$ is non-decreasing and this is the same thing to show that

$$\gamma'\left(\frac{\mathcal{R}(X + X')}{2}\right) \leq \gamma'\left(\frac{\mathcal{R}(X) + \mathcal{R}(X')}{2} \right).$$ We set $S:=X+X'$. Using the concavity of $\gamma$, we have 

\begin{alignat*}{1}
\gamma' \left(\frac{\mathcal{R}(X) + \mathcal{R}(X')}{2} \right) & \leq \frac{1}{2} \left[ \gamma' \Big(\mathcal{R}(X) \Big) + \gamma' \left(\mathcal{R}(X') \right) \right] 
\\
& = \frac{1}{2(1-\alpha)} \mathbb{E} \left(\gamma'(X) \mathbf{1}_{X \geq q_{\alpha}^{X}} + \gamma'(X') \mathbf{1}_{X' \geq q_{\alpha}^{X'}} \right)\\
\end{alignat*}

where we still denote $\mathcal{R}(V)=Q_{\alpha}^{d\gamma}(V)$ when $V$ is a random variable. 
But

$$ \gamma' \left(\frac{\mathcal{R}(S)}{2} \right)= \frac{1}{2(1-\alpha)} \mathbb{E} \left(\gamma'(S) \mathbf{1}_{S \geq q_{\alpha}^{S}}\right). $$

So we want to show that 

\[\mathbb{E} \left(\gamma'(X)  \mathbf{1}_{X \geq q_{\alpha}^{X}} + \gamma'(X')  \mathbf{1}_{X' \geq q_{\alpha}^{X'}} - \gamma'(S)  \mathbf{1}_{S \geq q_{\alpha}^{S}} \right) \geq 0.\]

The sud-additivity hypothesis on $\gamma$ allows us to use the same argument as in \cite{coherent} for the classical superquantile :

\begin{alignat*}{1}
 & \mathbb{E} \Big(\gamma'(X)  \mathbf{1}_{X \geq q_{\alpha}^{X}} + \gamma'(X')  \mathbf{1}_{X' \geq q_{\alpha}^{X'}} - \gamma'(S)  \mathbf{1}_{S \geq q_{\alpha}^{S}} \Big)\\
& \geq \mathbb{E} \Big(\gamma'(X)  \mathbf{1}_{X \geq q_{\alpha}^{X}} + \gamma'(X')  \mathbf{1}_{X' \geq q_{\alpha}^{X'}} - \gamma'(X)  \mathbf{1}_{S \geq q_{\alpha}^{S}} - \gamma'(X')  \mathbf{1}_{S \geq q_{\alpha}^{S}} \Big)\\
& \geq \gamma'(q_{\alpha}^{X}) \mathbb{E} \Big(\mathbf{1}_{X \geq q_{\alpha}^{X}} - \mathbf{1}_{S \geq q_{\alpha}^{S}} \Big) + \gamma'(q_{\alpha}^{X'} ) \mathbb{E} \Big(\mathbf{1}_{X' \geq q_{\alpha}^{X'}} - \mathbf{1}_{S \geq q_{\alpha}^{S}} \Big) \\
& = 0. \\
\end{alignat*}

\vspace{0.3cm}

Finally, we show the closeness under the same assumption as just before.  Let be $(X_{h})_{h>0}$ satisfying the hypothesis. By subadditivity we have

$$\mathcal{R}(X) \leq \mathcal{R}(X_h) + \mathcal{R}(X_{h}-X) \leq 0 + \mathcal{R}(X_{h}-X).$$

Then denoting $Y_h=X_h-X$, it is enough to show that $$Y_{h} \underset{L^{2}, \, n \to + \infty}{\longrightarrow} 0 \Longrightarrow \mathcal{R}(Y_{h}) \underset{n \to + \infty}{\longrightarrow} 0 $$ to conclude. Thanks to the concavity of $\gamma'$ we can use Jensen inequality for conditional expectation

$$ \gamma'^{-1} \left[\mathbb{E} \Big( \gamma'(Y_{h}) | Y_{h}\geq F_{Y_h}^{-1} (\alpha) \Big) \right] \leq \mathbb{E}\left(Y_{h} |Y_{h} \geq F_{Y_{h}}^{-1}(\alpha)\right).$$

We conclude with Cauchy-Schwartz inequality 

$$\frac{\mathbb{E}\left(Y_{h} \mathbf{1}_{Y_{h} \geq F_{Y_{h}}^{-1}(\alpha)}\right)}{1-\alpha} \leq \mathbb{E}((Y_{h})^{2})^{\frac{1}{2}} \frac{\mathbb{E} \left(\mathbf{1}^2_{Y_{h} \geq F_{Y_{h}^{-1}(\alpha)} }\right)^{\frac{1}{2}}}{1- \alpha}= ||Y_{h}||^{2} \sqrt{1-\alpha}  \underset{h \to 0}{\longrightarrow} 0.$$

\end{proof}

\subsection{Proof of Proposition \ref{superasympto} : asymptotic behavior of the plug-in estimator of the superquantile}

\subsection{Mathematical tools}

We first give some technical or classical results that we will use in the forthcoming proofs.

\subsubsection{Ordered statistics and Beta function}
\label{repres}

Let us recall some results on ordered statistics (see \cite{order}).
First of all, let $(U_{i})_{i=1 \dots n}$ be an independant sample from the uniform distribution on $[0, 1]$. Then, 

\begin{equation}
\begin{aligned}
\label{consistanceuniforme}
U_{(n)} \underset{a.s}{\longrightarrow}1.
\end{aligned}
\end{equation}

and

\begin{equation}
\begin{aligned}
\label{TCLuniforme}
n \left(1-U_{(n)} \right) \underset{\mathcal{L}}{\longrightarrow}W
\end{aligned}
\end{equation}

where $W$ has an exponential distribution of parameter 1.

Let now $(Y_{i})_{i=1 \dots n+1}$ be an independent sample from the standard exponential distribution. It's well known that 

\begin{equation}
\label{representation}
U_{(i)}:= \displaystyle\sum_{j=1}^{i}Y_{j}\left(\displaystyle\sum_{j=1}^{n+1}Y_{j}\right)^{-1}
\end{equation}

has the same distribution as the $i^{th}$ ordered statistics  of an i.i.d sample of size $n$ uniformly distributed on $[0, 1]$, that is Beta distribution of parameters $i$ and $n-i+1$ denoted $\mathcal{B}(i, n-i+1)$. 

It is also known that when $(X_i)_{1\leq i \leq n}$ is a sample of cumulative distribution $F_X$, this equality in law holds

\begin{equation}
\label{loi}
X_{(i)} \stackrel{\mathcal{L}}{=}  F^{-1}_X(U_{(i)}).
\end{equation}

Recall that $\mathcal{B}(a, b)$ distribution has the following density

$$f_{\mathcal{B}(a, b)}(x) = \frac{x^{a-1}(1-x)^{b-1}}{B(\alpha, \beta)}\mathbf{1}_{x \in [0, 1]}, \, \, a, b>0 $$

where 

\begin{equation}
B(a, b)= \int_0^1 t^{a-1}(1-t)^{b-1} dt = \frac{\Gamma(a) \Gamma(b)}{\Gamma(a+b)}.
\label{defbeta}
\end{equation}

A classical property of the Beta function is 

\begin{equation}
\forall (x, y) \in \mathbb{R}^+, \, \, B(x+1, y) = \frac{x}{x+y} B(x, y).
\label{betasym}
\end{equation}

Generalizing the definition of the factorials, we set for $n \in \mathbb{N}^*$

$$\left(n- \frac{1}{2}\right)! := \left(n- \frac{1}{2}\right) \left(n- 1-  \frac{1}{2}\right) \dots \left(\frac{1}{2}\right),$$

we have for $i \in \mathbb{N}^*, n\geq i+2$ 

\begin{equation}
B\left(i, n-i- \frac{5}{2} +1\right) = \frac{(i-1)! \left(n-i-2- \frac{1}{2}\right)!}{\left(n-2- \frac{1}{2}\right)!},
\label{beta}
\end{equation}

\begin{equation}
\left(n- \frac{1}{2} \right) ! = \frac{(2n)!}{(2^n)^2 n!}.
\label{factoriel}
\end{equation}

Indeed, Equation \eqref{beta} comes directly from the Definition \eqref{defbeta} and to see Equation \eqref{factoriel}, we fix $k= \frac{1}{2}$ and notice that

$$2^n(n-k)! = (2n-1)(2n-3) \dots 3 \times 1 =
\frac{(2n)!}{2n(2n-2) \dots 6 \times 4 \times 2} 
= \frac{(2n)!}{2^n n!} .$$

Moreover, the cumulative distribution function of the Beta distribution is the regularized incomplete Beta function $I_x(a, b)$. This function satisfies, when $a$ and $b$ are positive integers

\begin{equation*}
\begin{aligned}
I_x(a, b) = \mathbb{P} \left( \mathcal{B}(a+b-1,x) \geq a \right)
\end{aligned}
\end{equation*}

and then, Bernstein inequality for Bernoulli distribution (see for example Theorem 8.2 of \cite{livreaurelien}) gives

\begin{equation}
\label{devincomplete}
\begin{aligned}
I_x(\alpha, \beta) \leq \exp \left( - \frac{3}{8} (a+b-1) x\right)
\end{aligned}
\end{equation}

as soon as $\frac{a}{n} \geq 2 x$.

\subsubsection{Technical lemma}

\begin{lemma}
\label{lemmetechnic}

Let $\delta >1$ and $\beta \in ]0, 1[$.
Then $n^{-1} \displaystyle\sum_{i= \lfloor n \beta \rfloor}^{n-1} \left(1 - \frac{i}{n+1}\right)^{-\delta}= \mathcal{O}\left(\sqrt{n}\right)$ if and only if $\delta \leq \frac{3}{2}.$
\end{lemma}

\begin{proof}
Let $\delta>1$.
We have to characterize the $\delta$ for which 

$n^{-\frac{3}{2}}\displaystyle\sum_{i= \lfloor n \beta \rfloor}^{n-1} \left(1 - \frac{i}{n+1}\right)^{-\delta}$ is bounded when $n$ goes to infinity. Let us make the index change $j:=n+1-i$. Those sums become 

$$n^{-\frac{3}{2}}\displaystyle\sum_{j=2 }^{n+1- \lfloor n \beta \rfloor} \left(1-\frac{n+1-j}{n+1} \right)^{-\delta} = \frac{n^{\frac{-3}{2}}}{(n+1)^{-\delta}}\displaystyle\sum_{j=2 }^{n+1- \lfloor n \beta \rfloor} \frac{1}{j^{\delta}} \underset{n \to + \infty}{\sim} n^{\delta - \frac{3}{2}} \left(\zeta(\delta)-1 \right),$$

where $\zeta$ denote the Zeta function. Then, if $\delta > \frac{3}{2}$, $\zeta(\delta)$ is finite and the behaviour of the sum is the same as the one of $n^{-\frac{3}{2} + \delta}$ which diverges to $+ \infty$. On the contrary, if $1 < \delta \leq \frac{3}{2}$, $\zeta(\delta)$ is still finite but $n^{-\frac{3}{2} + \delta}$ is bounded and so does the sum. 
\end{proof}

\subsubsection{A corollary of Lindenberg-Feller theorem}

To prove the asymptotic normality, we use a central limit theorem which is a corollary of the Lindeberg-Feller theorem (see lemma 1 in \cite{uniform}). 

\begin{proposition}
\label{corlinde}
Let $(Y_{1}, \dots, Y_{n})$ be an independent sample of exponential variables of parameter 1 and $(\alpha_{j, n})_{j \leq n, \, n\geq 2}$ be a triangle array of real numbers. 

If $Q_{n}= n^{-1} \displaystyle\sum_{j=1}^{n} \alpha_{j, n} (Y_{j}-1)$ and $\sigma_{n}^{2}=\frac{1}{n} \displaystyle\sum_{j=1}^{n} \alpha_{j, n}^{2}$, then
 
 $$ \frac{\sqrt{n}Q_{n}}{\sigma_{n}} \Longrightarrow \mathcal{N}(0, 1)$$
 
 if and only if $\max_{1 \leq j \leq n} | \alpha_{j, n} | = o(n^{\frac{1}{2}} \sigma_{n})$.

If furthermore $\sigma_{n}$ converges in probability to $\sigma$ then by Slutsky's lemma 

 $$ \sqrt{n}Q_{n} \Longrightarrow \mathcal{N}(0, \sigma^{2}).$$
\end{proposition}

\subsection{Proof of i) of Theorem \ref{superasympto} : consistency of the plug-in estimator (\ref{estimateurSQ})}

\begin{proof}

We aim to show consistency of the estimator (\ref{estimateurSQ}).
Let us first notice that 

$$Q_{\alpha}  = \frac{\mathbb{E} \left(X \mathbf{1}_{X \geq F_X^{-1}(x)(\alpha)} \right)}{1-\alpha} 
 = \frac{\int_{\mathbb{R}}x \mathbf{1}_{X \geq F_X^{-1}(\alpha)} f_X(x) dx}{1- \alpha} = \frac{\int_{\alpha}^1 F_X^{-1}(y) dy}{1-\alpha}. $$

Thus, we need to show that

$$\frac{1}{n} \sum_{i = \lfloor n \alpha \rfloor }^n X_{(i)} - \int_{\alpha}^1 F_X^{-1}(y) dy \underset{n \to + \infty}{\longrightarrow} 0 \, \,  a.s.$$ 
In the sequel we omit the index $X$ in $F_X^{-1}$ because there is no ambiguity. Let us introduce the two following quantities and show their convergence to 0 in probability.   $$A_n = \frac{1}{n} \sum_{i = \lfloor n \alpha \rfloor }^n X_{(i)} - \frac{1}{n} \sum_{i = \lfloor n \alpha \rfloor }^n F^{-1}\left(\frac{i}{n+1}\right), $$ and $$B_n = \frac{1}{n} \sum_{i = \lfloor n \alpha \rfloor }^n F^{-1}\left(\frac{i}{n+1}\right) -  \int_{\alpha}^1 F^{-1}(y) dy.$$

Let us first deal with $A_n$. 


We know by (\ref{loi}) that $X_{(i)} \sim F_X^{-1}\left( U_{(i)} \right)$ where $U_{(i)}$ is distributed like the $i^{th}$ ordered statistic of a uniform sample. Thus, defining $U_{(i)}$ with distribution $\mathcal{B}(i, n+1-i)$, it holds that,

$$A_n=\frac{1}{n} \displaystyle\sum_{i= \lfloor n \alpha \rfloor}^n\left(  F^{-1}\left(U_{(i)}\right) - F^{-1}\left(\frac{i}{n+1} \right) \right).$$

We now need to split the sum in two parts. First, let us deal with the last term in the sum (which gives actually the biggest contribution). By the mean value theorem, there exists $w_n \in \left[ U_{(n)}, \frac{n}{n+1} \right]$ (where we use non-oriented interval) such that 

$$\frac{1}{n} \left( F^{-1}(U_{(n)}) - F^{-1}\left( \frac{n}{n+1} \right) \right) = \frac{1}{n} \left( U_{(n)} - \frac{n}{n+1} \right)l(w_n).$$

Since (\ref{consistanceuniforme}) holds, and thanks to assumption \textbf{H3}, there exists a constant $C_1$ such that for $n$ big enough

\begin{equation}
\begin{aligned}
\label{pareil}
\left| \frac{1}{n} \left( F^{-1}(U_{(n)}) - F^{-1}\left( \frac{n}{n+1} \right) \right) \right|& \leq \frac{C_1}{n} \frac{1}{(1-w_n)^{2-\epsilon_l}} \left|U_{(n)} - \frac{n}{n+1} \right| \\
& \leq \frac{1}{n} \frac{C_1}{\left(1- \frac{n}{n+1} \right)^{2-\epsilon_l}} \left|U_{(n)} - \frac{n}{n+1} \right| + \frac{1}{n} \frac{C_1}{\left( 1- U_{(n)} \right)^{2-\epsilon_l}} \left|U_{(n)} - \frac{n}{n+1} \right|\\
& \leq \frac{C_1}{n} (n+1)^{2-\epsilon_l}|U_{(n)} -1 | + \frac{C_1}{n} (n+1)^{2- \epsilon} \left| \frac{n}{n+1}-1 \right| \\
&  + \frac{C_1}{n} \frac{1}{\left( 1- U_{(n)} \right)^{2 - \epsilon_l}} \left| U_{(n)} -1 \right| + \frac{C_1}{n} \frac{1}{\left( 1- U_{(n)} \right)^{2 - \epsilon_l}} \left| \frac{n}{n+1} -1 \right| \\
& \leq C_1W_n\frac{(n+1)^{2-\epsilon_l}}{n^2} + C_1\frac{(n+1)^{2-\epsilon_l}}{n(n+1)} + \frac{C_1}{W_n^{1-\epsilon_l} n^{\epsilon_l}} + \frac{C_1}{W_n^{2-\epsilon_l}} \frac{n^{1-\epsilon_l}}{n+1} \\
\end{aligned}
\end{equation}

where $W_n=n\left(1- U_{(n)} \right)$. Thanks to (\ref{TCLuniforme}) and the Slutsky lemma, we have shown the convergence in probability to 0 of this term. Terms for $i=n-1$ and $i=n-2$ can be treated exactlty in the same way.

\vspace{0.3cm}

Let us now deal with the remaining sum (for i from $\lfloor n \alpha \rfloor$ to $n-3$) that we still denote $A_n$. By the mean value theorem, we will upper-bound $A_n$ by a quantity depending on $l$. Since we know the behaviour of $l$ on every compact set and in the neighborhood of 1, we begin by showing that when $n$ becomes big enough, $U_{(\lfloor n \alpha \rfloor)}$ is far away from $0$. Let $\frac{\alpha}{2}\geq \epsilon> 0$ be a positive real number. We have,

$$A_n= A_n \mathbf{1}_{U_{( \lfloor n \alpha \rfloor ) }< \epsilon} + A_n \mathbf{1}_{U_{( \lfloor n \alpha \rfloor ) }\geq \epsilon}.$$

Then, for $\eta'>0$, thanks to (\ref{devincomplete}) of recallings and because for $n$ big enough $\frac{\lfloor n \alpha \rfloor}{n} \geq \frac{\alpha}{2} \geq \epsilon$, we get

\begin{equation}
\label{zero}
\begin{aligned}
\P( A_n \mathbf{1}_{U_{( \lfloor n \alpha \rfloor )} < \epsilon} > \eta') \leq \P(U_{( \lfloor n \alpha \rfloor )} < \epsilon) \leq \exp \left( - \frac{3n \epsilon}{8} \right). \\
\end{aligned}
\end{equation}

Then, it is enough to show the convergence to 0 in probability of $A_n':=A_n \mathbf{1}_{U_{( \lfloor n \alpha \rfloor )} \geq \epsilon}$. Let us show its converges to 0 in $L^1$.

By the mean value theorem, we know that there exists $w_i \in ]U_{(i)}, \frac{i}{n+1}[$ (we do not know if $U_{(i)}$ is smaller or bigger than $\frac{i}{n+1}$ but in the sequel we still denote the segment bewteen $U_{(i)}$ and $\frac{i}{n+1}$ in this way) such that

\begin{equation*}
\begin{aligned}
\mathbb{E} ( |A_n'|) \leq \frac{1}{n} \displaystyle\sum_{i= \lfloor n-3 \alpha \rfloor}^n \mathbb{E} \left[ \Big|U_{(i)} - \frac{i}{n+1} \Big| l(w_i) \mathbf{1}_{U_{(\lfloor n \alpha \rfloor)} \geq \epsilon} \right] .\\
\end{aligned}
\end{equation*}

But, for all $i \in \{ \lfloor n \alpha \rfloor, \dots, n-3\}$, we get 
$U_{(i)} \geq U_{(\lfloor n \alpha \rfloor)} \geq \epsilon$. Moreover, for $n$ big enough, we get, $\frac{i}{n+1} \geq \frac{\lfloor n \alpha \rfloor}{n+1} \geq \frac{\alpha}{2} \geq \epsilon$. Then,

$$\forall i \in \{ \lfloor n \alpha \rfloor, \dots, n-3\}, \, \, w_i \in [\epsilon, 1 [.$$

Then, according to \textbf{H3} and Remark \ref{appelle} (here since we do not deal with biggest terms in the sum we only need a weaker assumption than \textbf{H3}), we get

\begin{itemize}
\item for any arbitrary $\eta >0$, there exists $\xi_{\eta} >0$ such that 

$$\forall t \in ]1 - \xi_{\eta}, 1[, \, \, l(t) \leq \frac{\eta}{(1-t)^2}$$

\item because $l$ is continuous, there exists a constant $C$ such that 

$$\forall t  \in [\epsilon, 1- \xi_{\eta}], \, \, l(t) \leq C.$$

\end{itemize}

Let us then look at the sum on two differents events. 

%
%

\begin{equation*}
\begin{aligned}
\mathbb{E}(|A_n'|) & \leq \frac{1}{n} \displaystyle\sum_{i= \lfloor n \alpha \rfloor}^{n-3} \mathbb{E} \left[ \Big|U_{(i)} - \frac{i}{n+1} \Big| l(w_i) \mathbf{1}_{U_{(\lfloor n \alpha \rfloor)} \geq \epsilon} \mathbf{1}_{w_i >1-\xi_{\eta}}\right] \\
& + \frac{1}{n} \displaystyle\sum_{i= \lfloor n \alpha \rfloor}^{n-3} \mathbb{E} \left[ \Big|U_{(i)} - \frac{i}{n+1} \Big| l(w_i) \mathbf{1}_{U_{(\lfloor n \alpha \rfloor)} \geq \epsilon} \mathbf{1}_{w_i \leq 1-\xi_{\eta}}\right] \\
& := T_n^1+T_n^2\\
\end{aligned}
\end{equation*}

But,

\begin{equation*}
\begin{aligned}
T_n^{1} & \leq \frac{1}{n} \displaystyle\sum_{i= \lfloor n \alpha \rfloor}^{n-3} \mathbb{E} \left[ \Big|U_{(i)} - \frac{i}{n+1} \Big| l(w_i) \mathbf{1}_{\epsilon \leq w_i \leq 1-\xi_{\eta}}\right] \\
& \leq \frac{1}{n} \displaystyle\sum_{i= \lfloor n \alpha \rfloor}^{n-3} \mathbb{E} \left[ \Big|U_{(i)} - \frac{i}{n+1} \Big| C \right]\\
&\leq \frac{1}{n} \displaystyle\sum_{i= \lfloor n \alpha \rfloor}^{n-3} C \sqrt{\mbox{Var}(U_{(i)})} \\
& \leq \frac{1}{n} \displaystyle\sum_{i= \lfloor n \alpha \rfloor}^{n-3} C \sqrt{\frac{i(n-i+1)}{(n+1)^2(n+2)}} \\
&\sim \frac{1}{\sqrt{n}}  \frac{1}{n} \displaystyle\sum_{i= \lfloor n \alpha \rfloor}^{n-3} C \sqrt{\frac{i}{n}\left(1-\frac{i}{n}\right)} = \mathcal{O} \left( \frac{1}{\sqrt{n}} \right).\\
\end{aligned}
\end{equation*}

thanks to convergence of Rieman's sum of the continuous function $x \mapsto \sqrt{x(1-x)}$ on $[\alpha, 1]$. Then, this terms goes to 0. Let us conclude by dealing with the last term. 

\begin{equation}
\label{eqmax}
\begin{aligned}
T_n^{2} & \leq \frac{1}{n} \displaystyle\sum_{i= \lfloor n \alpha \rfloor}^{n-3} \mathbb{E} \left[ \Big|U_{(i)} - \frac{i}{n+1} \Big| l(w_i) \mathbf{1}_{w_i > 1-\xi_{\eta}}\right] \\
&\leq \frac{1}{n} \displaystyle\sum_{i= \lfloor n \alpha \rfloor}^{n-3} \mathbb{E} \left[ \Big|U_{(i)} - \frac{i}{n+1} \Big| \frac{\eta}{(1-w_i)^2} \right]\\
& \leq  \frac{1}{n} \displaystyle\sum_{i= \lfloor n \alpha \rfloor}^{n-3}  \mathbb{E} \left[ \Big|U_{(i)} - \frac{i}{n+1} \Big| \max \left( \frac{\eta}{(1-U_{(i)})^2},\frac{\eta}{(1-\frac{i}{n+1})^2}\right) \right].\\ 
\end{aligned}
\end{equation}

But, by the Cauchy-Schwartz inequality,

\begin{equation*}
\begin{aligned}
 \frac{1}{n} \displaystyle\sum_{i= \lfloor n \alpha \rfloor}^{n-3}  \mathbb{E} \left[ \big|U_{(i)} - \frac{i}{n+1} \Big|  \frac{\eta}{(1-U_{(i)})^2} \right]& \leq \frac{\eta}{n} \displaystyle\sum_{i= \lfloor n \alpha \rfloor}^{n-3} \sqrt{\mbox{Var}(U_{(i)}) \mathbb{E} \left( \frac{1}{(1-U_{(i)})^4}\right)} \\
\end{aligned}
\end{equation*}

Since $i < n-2$, all the terms of the sum can be expressed using Beta functions and then the expectation is finite (and this is for this reason that we ca not include biggest terms $i=n$, $i=n-1$ and $i=n-2$ in this reasonning). Indeed,

\begin{small}
\begin{equation*}
\begin{aligned}
 \frac{\eta}{n} \displaystyle\sum_{i= \lfloor n \alpha \rfloor}^{n-3} \sqrt{\mbox{Var}(U_{(i)}) \mathbb{E} \left( \frac{1}{(1-U_{(i)})^4}\right) } & \leq \frac{\eta}{n} \displaystyle\sum_{i= \lfloor n \alpha \rfloor}^{n-3} \sqrt{\frac{i(n-i+1)}{(n+1)^2(n+2)}} \sqrt{\frac{B(i, n+1-i) }{\int_0^1 x^{i-1}(1-x)^{n+1-i-4} dx}} \\
& = \frac{\eta}{n} \displaystyle\sum_{i= \lfloor n \alpha \rfloor}^{n-3} \sqrt{\frac{i(n-i+1)}{(n+1)^2(n+2)}} \sqrt{\frac{B(i, n+1-i-3)}{B(i, n+1-i)}} \\
& = \frac{\eta}{n} \displaystyle\sum_{i= \lfloor n \alpha \rfloor}^{n-3} \sqrt{\frac{n(n-1)(n-2)}{(n+1)^2(n+2)}\frac{i(n-i+1)}{(n-i)(n-1-i)(n-2-i)}} \\
\end{aligned}
\end{equation*}
\end{small}
 
where we used recallings on Beta function. The final term has the same behaviour when $n$ goes to $+ \infty$ that

$$ \eta \frac{1}{n^{\frac{3}{2}}} \displaystyle\sum_{i= \lfloor n \alpha \rfloor}^{n-3} \frac{1}{\left( 1- \frac{i}{n} \right)^{\frac{3}{2}}}:= \eta \times V_n$$

Since Lemma \ref{lemmetechnic} implies that $V_n$ is bounded independently of $\eta$ and because $\eta$ is arbitrary small, we have shown the converge to 0. Then the term in $U_{(i)}$ on the maximum of Equation (\ref{eqmax}) converges to 0. The second term can be computed in the same way

\begin{equation*}
\begin{aligned}
 \frac{\eta}{n} \displaystyle\sum_{i= \lfloor n \alpha \rfloor}^{n-3} \sqrt{\mbox{Var}(U_{(i)}) \mathbb{E} \left( \frac{1}{(1-\frac{i}{n+1})^4}\right) }\sim \frac{1}{\sqrt{n}}\frac{\eta}{n} \displaystyle\sum_{i= \lfloor n \alpha \rfloor}^{n-3} \frac{1}{\left( 1- \frac{i}{n+1} \right)^{\frac{3}{2}}}. \\
 \end{aligned}
\end{equation*}

which also converges to 0 thanks to Lemma \ref{lemmetechnic}. Finally, $A_n'$ converges to 0 in $L^1$ and so in probability. So is $A_n$.

\vspace{0.3cm}

Let us now study the term $$B_n= \frac{1}{n} \displaystyle\sum_{i = \lfloor n \alpha \rfloor }^n F^{-1}\left(\frac{i}{n+1}\right) -  \int_{\alpha}^1 F^{-1}(y) dy.$$ We will show that this term converges to 0 thanks to a generalized Riemann sum convergence due to the monotonicity of $F^{-1}$.

\begin{remark}
\label{help}
To begin with, for $\epsilon >0$ it is easy to show that

$$B_n^{\epsilon}= \frac{1}{n} \displaystyle\sum_{i= \lfloor n \alpha \rfloor}^{\lfloor n (1 - \epsilon) \rfloor} F^{-1}\left( \frac{i}{n+1} \right) - \int_{\alpha}^{1- \epsilon} F^{-1}(y) dy$$

converges to 0. Indeed, it is the convergence of the Riemann sum for the continuous function $F^{-1}$.

\end{remark}

Let us fix $\epsilon  >0$. According to the previous remark, we split the forthcoming sum in two parts.

$$B_n= \frac{1}{n+1} \sum_{i = \lfloor n \alpha \rfloor }^{\lfloor n (1- \epsilon) \rfloor} F^{-1}\left(\frac{i}{n+1}\right) + \frac{1}{n+1} \sum_{\lfloor n (1- \epsilon) \rfloor +1}^{n-1} F^{-1}\left(\frac{i}{n+1}\right) := S_n^1 + S_n^2. $$

Since the quantile function is non-decreasing on $[\alpha, 1]$, we have :

\begin{alignat*}{1}
& \int_{\frac{\lfloor n \alpha \rfloor -1}{n+1}}^{\frac{\lfloor n (1- \epsilon) \rfloor}{n+1}} F^{-1}(t) dt + \int_{\frac{\lfloor n (1- \epsilon) \rfloor}{n+1}}^\frac{n-1}{n+1} F^{-1}(t) dt:=C_n^1+C_n^2 \\
& \leq  \frac{1}{n+1} \sum_{i = \lfloor n \alpha \rfloor }^{\lfloor n (1- \epsilon) \rfloor} F^{-1}\left(\frac{i}{n+1}\right) + \frac{1}{n+1} \sum_{\lfloor n (1- \epsilon) \rfloor +1}^{n-1} F^{-1}\left(\frac{i}{n+1}\right) :=S_n^1+S_n^2\\
& \leq  \int_{\frac{\lfloor n \alpha \rfloor }{n+1}}^{\frac{\lfloor n (1- \epsilon) \rfloor}{n+1}} F^{-1}(t) dt + \int_{\frac{\lfloor n (1- \epsilon) \rfloor}{n+1}}^\frac{n}{n+1} F^{-1}(t) dt:=D_n^1+D_n^2. \\
\end{alignat*}

Then, we have :

$$(C_n^1-S_n^1) + (C_n^2 - D_n^2) \leq (S_n^2 - D_n^2) \leq (D_n^1 - S_n^1) $$

If we show that $C_n^2-D_n^2$ converge to 0, we can conclude using comparison theorem, beacause the convergence of $D_n^1 - S_n^1$ and $C_n^1 - S_n^1$ to 0 is true thanks to the Remark \ref{help}. Let us then show this convergence 

As in the neighborhood of 1, $l(t)=o\left((1-t)^{-2}\right)$ (Remark \ref{appelle}), we also have $F^{-1}(t) = o \left( (1-t)^{-1} \right). $

Then, for $\epsilon >0$, there exist $N$ such that for $n \geq N$ :

$$C_n^2-D_n^2  =  - \int_{\frac{n-1}{n+1}}^{\frac{n}{n+1}} F^{-1}(t) dt  
\leq  \epsilon \int_{\frac{n-1}{n+1}}^{\frac{n}{n+1}} \frac{1}{1-t} dt 
= \epsilon \ln \left(2\right). $$

Finally, $S_n^2- D_n^2$ converges to 0 a.s. So that, the same holds for $B_n$.

We have shown that $A_n+B_n$ converge to 0 in probability. So under our hypothesis, the superquantile is consistent in probability. 

\vspace{0.3cm}

\begin{remark}
\label{remarkordresup}

Using the same arguments, we can show that under stronger hypothesis on the quantile function $F^{-1}(t)=o \left(\frac{1}{(1-t)^{\frac{1}{2}} }\right)$ (that is the case in ii) of Proposition \ref{superasympto}), we have 

$$- \int_{\frac{n-1}{n+1}}^{\frac{n}{n+1}} F^{-1}(t) dt  \leq  \epsilon \int_{\frac{n-1}{n+1}}^{\frac{n}{n+1}} \frac{1}{(1-t)^{\frac{1}{2}}} dt = \epsilon -2(1- \sqrt{2}) \frac{1}{\sqrt{n}}. $$

Then

$$ \sqrt{n} \left( \frac{1}{n} \sum_{i = \lfloor n \alpha \rfloor }^n F^{-1}(\frac{i}{n+1}) -  \int_{\alpha}^1 F^{-1}(y) dy \right) \underset{n\to+\infty}{\longrightarrow} 0.$$

We will use this result in the next part. 
\end{remark}

\end{proof}

\subsubsection{Proof of ii) of Proposition \ref{superasympto} : asymptotic normality of the plug-in estimator (\ref{estimateurSQ})}

Let us prove the asymptotic normality of the estimator of the superquantile. To begin with, we can make some technical remarks. 

\begin{remark}
\label{primitive}
The assumption on $L$ implies that there exists $\epsilon_l>0$ and $\epsilon_{F^{-1}}>0$ such that $ l(t)=O\left((1-t)^{- \frac{3}{2}+ \epsilon_l}\right)$, and $F^{-1}(t)=O\left((1-t)^{- \frac{1}{2}+ \epsilon_{F^{-1}}}\right).$ It also implies that in the neighborhood of 1, $L(t)=o\left((1-t)^{-\frac{5}{2}}\right)$.
\end{remark}

\begin{proof}

The proof stands in three steps. First we reformulate and simplify the problem and apply the Taylor Lagrange formula. Then, we show that the second order term converges to 0 in probability. In the third step, we identify the limit of the first order term. 

\vspace{0.3cm}

\textbf{Step 1 : Taylor-Lagrange formula}

\vspace{0.3cm}
Let us first omit $\alpha^{-1}$.
We have to study the convergence in distribution of 

$$ \sqrt{n} \left( \frac{1}{n} \displaystyle\sum_{i=\lfloor n \alpha \rfloor }^{n} X_{(i)} - \int_{\alpha}^{1} F^{-1}(y) dy  \right).$$

We have already noticed (Remarks \ref{remarkordresup} and \ref{primitive}) that

$$\sqrt{n} \left[  \frac{1}{n} \displaystyle\sum_{i=\lfloor n \alpha \rfloor }^{n} F^{-1}\left(\frac{i}{n+1}\right) - \int_{\alpha}^{1} F^{-1}(y) dy \right]  \underset{n \to + \infty}\longrightarrow 0 .$$

Thus, Slutsky's lemma, allows us to study only the convergence in law of

$$\sqrt{n} \left[ \frac{1}{n} \displaystyle\sum_{i=\lfloor n \alpha \rfloor }^{n} X_{(i)} - \frac{1}{n} \displaystyle\sum_{i=\lfloor n \alpha \rfloor }^{n} F^{-1}\left(\frac{i}{n+1}\right) \right]. $$

The quantile function $F^{-1}$ is two times differentiable so that we may apply the first order Taylor-Lagrange formula. Using the same argument as in the proof of i), we introduce $U_{(i)}$ a random variable distributed as a $\mathcal{B}(i, n+1-i)$. Considering an equality in law we then have 

\begin{alignat*}{1}
\sqrt{n}\left(\frac{1}{n} \displaystyle\sum_{i=\lfloor n \alpha \rfloor +1}^{n} \left[ X_{(i)}-F^{-1}\left(\frac{i}{n+1}\right) \right] \right) & \stackrel{\mathcal{L}}{=}  \sqrt{n} \left[\frac{1}{n} \displaystyle\sum_{i=\lfloor n \alpha \rfloor +1}^{n} \left( U_{(i)} - \frac{i}{n+1} \right) \frac{1}{f\Big(F^{-1}(\frac{i}{n+1})\Big)}\right] \\
& +  \frac{1}{\sqrt{n}} \displaystyle\sum_{i=\lfloor n \alpha \rfloor +1}^{n} \left[  \int_{\frac{i}{n+1}}^{U_{(i)}}  \frac{f'(F^{-1}(t))}{ \left[f(F^{-1}(t)) \right]^{3}}  \left(U_{(i)} - t \right) dt \right]. \\
\end{alignat*}

Let us call $\sqrt{n} Q_{n}$ the first-order term and $R_{n}$ the second-order one. 

\vspace{0.3cm}

\textbf{Step 2 : The second-order term converges to 0 in probability}

\vspace{0.3cm}

Let us show that  $R_{n}$ converges to 0 in probability. We will use the same decomposition as for $A_n$. First, we deal with the last term : $i=n$. Still using (\ref{consistanceuniforme}) and \textbf{H3} we have the existence of a constant $C_2$ such that for $n$ big enough,

\begin{small}
\begin{equation*}
\begin{aligned}
\left|\frac{1}{\sqrt{n}} \int_{\frac{n}{n+1}}^{U_{(n)}} L(t) \left( U_{(n)} -t \right) dt \right| & \leq \frac{1}{\sqrt{n}} \max \left( \frac{C_2}{\left( 1- U_{(n)} \right)^{\frac{5}{2}- \epsilon_L}}, \frac{C_2}{\left( 1- \frac{n}{n+1} \right)^{\frac{5}{2}- \epsilon_L}}\right) \frac{\left(U_{(n)} - \frac{n}{n+1}\right)^2}{2}\\
 & = \frac{1}{\sqrt{n}} \frac{\left(U_{(n)} - \frac{n}{n+1} \right)^2}{2} \frac{C_2}{\left(1-U_{(n)}\right)^{\frac{5}{2} - \epsilon_L}} + \frac{1}{\sqrt{n}} \frac{\left(U_{(n)} - \frac{n}{n+1} \right)^2}{2} \frac{C_2}{\left(1-\frac{n}{n+1}\right)^{\frac{5}{2} - \epsilon_L}}\\
\end{aligned}
\end{equation*}
\end{small}

which converges to 0 in probability exactly as in (\ref{pareil}) thanks to (\ref{TCLuniforme}) and Slutsky lemma. This is the same idea for the term $i=n-1$. 

\vspace{0.3cm}

Let us now deal with the remaining sum (for $i$ from $ \lfloor n \alpha \rfloor$ to $n-2$) that we still denote $R_n$. We use same kind of reasonning that for $A_n$. First, we still have for $\frac{\alpha}{2}\geq \epsilon >0$,

\begin{equation*}
\begin{aligned}
R_n= R_n\mathbf{1}_{U_{( \lfloor n \alpha \rfloor)} < \epsilon} + R_n\mathbf{1}_{U_{( \lfloor n \alpha \rfloor)} \geq \epsilon}
\end{aligned}
\end{equation*}

The first term converges in probability to 0 using the same argument as in (\ref{zero}). Let us then deal with the second term that we denote $R_n'$ and show its convergence in $L^1$. Since \textbf{H4} gives also informations on a neighborhood of 1 and on every compact set, we will use the same kind of argument as before (and so Remark \ref{appelle}). For $\eta>0$, there exists, thanks to \textbf{H4}, a real number $\xi_{\eta}$ such that 

\begin{itemize}
\item $\forall t \in ]1- \xi_{\eta}, 1[, |L(t)|\leq \frac{\eta}{(1-t)^{ \frac{5}{2}}}$. 

\item On $[\epsilon, 1- \xi_{\eta}]$ which is a compact set, the function $L$ is bounded by a constant $C_3$. 
\end{itemize}

Moreover, since $U_{( \lfloor n \alpha \rfloor)} \geq \epsilon$, we have already seen that 

$$ ]U_{(i)}, \frac{i}{n+1}[ \subset [\epsilon, 1[, \, \, (i= \lfloor n \alpha \rfloor,  \lfloor n \alpha \rfloor +1, \, \, \dots \, \, n).$$

Finally, we get, by denoting $m_i=\min\left( U_{(i)}, \frac{i}{n+1} \right)$ and $M_i=\max\left( U_{(i)}, \frac{i}{n+1} \right)$, $(i= \lfloor n \alpha \rfloor,  \lfloor n \alpha \rfloor +1, \, \, \dots \, \, n)$,

\begin{small}
\begin{equation*}
\begin{aligned}
\mathbb{E}(|R_n'|) & \leq \frac{1}{\sqrt{n}} \displaystyle\sum_{i= \lfloor n \alpha \rfloor}^{n-2} \mathbb{E} \left( \Big| \int_{U_{(i)}}^{\frac{i}{n+1}} \left( U_{(i)} -t \right) L(t)dt \Big| \right)\\
& \leq  \frac{1}{\sqrt{n}} \displaystyle\sum_{i= \lfloor n \alpha \rfloor}^{n-2} \mathbb{E} \left( \int_{m_i}^{1- \xi_{\eta}} \Big| U_{(i)} -t \Big| |L(t)| dt  \right) + \frac{1}{\sqrt{n}} \displaystyle\sum_{i= \lfloor n \alpha \rfloor}^{n-2} \mathbb{E} \left( \int_{1- \xi_{\eta}}^{M_i} \Big| U_{(i)} -t \Big| |L(t)| dt  \right) \\
& \leq  \frac{1}{\sqrt{n}} \displaystyle\sum_{i= \lfloor n \alpha \rfloor}^{n-2}  C_3 \mathbb{E} \left( \frac{\left( U_{(i)}-\frac{i}{n+1} \right)^2}{2} \right) + \frac{1}{\sqrt{n}} \displaystyle\sum_{i= \lfloor n \alpha \rfloor}^{n-2}  \mathbb{E} \left(\max \left( \frac{\eta}{\left( 1 - U_{(i)} \right)^{\frac{5}{2}}}, \frac{\eta}{\left( 1 - \frac{i}{n+1} \right)^{\frac{5}{2}}} \right) \frac{\left( U_{(i)} - \frac{i}{n+1}\right)^2}{2} \right) \\
& \leq \frac{1}{\sqrt{n}} \frac{C_3}{2(n+2)} \displaystyle\sum_{i= \lfloor n \alpha \rfloor}^{n-2} \frac{i}{n+1} \left( 1 - \frac{i}{n+1} + \frac{1}{n+1} \right) :=S_n^1 \\
&+ \frac{1}{\sqrt{n}} \displaystyle\sum_{i= \lfloor n \alpha \rfloor}^n  \mathbb{E} \left(\frac{\eta}{\left( 1 - U_{(i)} \right)^{\frac{5}{2}}} \frac{\left( U_{(i)} - \frac{i}{n+1}\right)^2}{2} \right) :=S_n^2 \\
& + \frac{1}{\sqrt{n}} \displaystyle\sum_{i= \lfloor n \alpha \rfloor}^{n-2} \mathbb{E} \left(\frac{\eta}{\left( 1 - \frac{i}{n+1}\right)^{\frac{5}{2}} } \frac{\left( U_{(i)} - \frac{i}{n+1}\right)^2}{2} \right):=S_n^3.\\
\end{aligned}
\end{equation*}
\end{small}

$S_n^1$ converges to 0 as the Riemann sum of the continuous function $x \mapsto x(1-x)$ multiplied by $n^{- \frac{1}{2}}$. We have also

\begin{equation*}
\begin{aligned}
S_n^3 &=  \frac{1}{\sqrt{n}} \displaystyle\sum_{i= \lfloor n \alpha \rfloor}^{n-2}  \frac{\eta}{\left( 1 - \frac{i}{n+1}\right)^{\frac{5}{2}}} \frac{i(n+1-i)}{(n+2)(n+1)^2} \\
\end{aligned}
\end{equation*}

which has the same behaviour when $n$ goes to $+\infty$ that

$$ \eta \frac{1}{\sqrt{n^{\frac{3}{2}}}}  \displaystyle\sum_{i= \lfloor n \alpha \rfloor}^{n-2} \frac{1}{\left(1- \frac{i}{n}\right)^{\frac{3}{2}}} := \eta V_n$$

which goes to 0 because thanks to Lemma \ref{lemmetechnic}, $V_n$ is bounded independently of $\eta$ and because $\eta$ is arbitrary small. Finally, to study the converges to 0 of $S_n^2$, we have to compute

\begin{alignat*}{1}
  \mathbb{E} \left(\frac{\left(U_{(i)} - \frac{i}{n+1}\right)^2}{\left(U_{(i)} - 1\right)^\frac{5}{2}} \right) & = \frac{1}{B(i, n+1-i)} \int_0^1 x^{i-1}(1-x)^{n-i-\frac{5}{2}} \left(x- \frac{i}{n+1} \right)^2 dx \\
 & = \frac{1}{B(i, n+1-i)} \Big( \int_0^1 x^{i+1}(1-x)^{n-i-\frac{5}{2}} dx -2 \frac{i}{n+1} \int_0^1 x^{i}(1-x)^{n-i-\frac{5}{2}} dx \\
 & + \left(\frac{i}{n+1} \right)^2 \int_0^1 x^{i-1}(1-x)^{n-i-\frac{5}{2}} \Big). \\
\end{alignat*}

Let us call this last quantity $I_n^i$. For $i < n-1$,

\begin{alignat*}{1}
I_n^i  & = \frac{1}{B(i, n+1-i)} \Big[ B\left(i+2, n-i-\frac{5}{2} +1\right) - 2 \frac{i}{n+1}  B \left(i+1, n-i-\frac{5}{2} +1 \right) \\
& + \left(\frac{i}{n+1}\right)^2 B \left(i, n-i-\frac{5}{2} +1 \right) \Big]. \\
\end{alignat*}

So that using \eqref{betasym} we obtain

\begin{alignat*}{1}
 I_n^i & = \frac{B \left(i, n-i-\frac{5}{2}+1\right)}{B(i, n+1-i)} \left[ \frac{i(i+1)}{\left(n-\frac{5}{2} +2\right)\left(n- \frac{5}{2} +1\right)} \right. \\
& - \left.2  \frac{i^2}{\left(n-\frac{5}{2} +1\right)(n+1)}  \frac{i(i+1)}{\left(n-\frac{5}{2} +2\right)\left(n- \frac{5}{2} +1\right)} 
+ \left(\frac{i}{n+1}\right)^2  \frac{i(i+1)}{\left(n-\frac{5}{2} +2\right)\left(n- \frac{5}{2} +1\right)} \right]. \\
\end{alignat*}

Let $E_n^i$ be such that $I_n^i= \frac{B(i, n-i-\frac{5}{2}+1)}{B(i, n+1-i)} E_n^i$. Expanding $E_i^n$ gives when $n$ goes to infinity 

$$E_i^n \sim \frac{1}{n} \frac{i}{n+1} \left( 1- \frac{i}{n+1} \right).$$

Let us study the term $\frac{B\left(i, n-i-\frac{5}{2} +1\right)}{B(i, n+1-i)}$. Using $\eqref{defbeta}$ and $\eqref{factoriel}$, we obtain 

\begin{alignat*}{1}
 \frac{B\left(i, n-i-\frac{5}{2} +1\right)}{B(i, n+1-i)} 
& = \frac{n!}{\left(n-2- \frac{1}{2}\right)!} \frac{\left(n-i-2- \frac{1}{2}\right)!}{(n-i)!} \\
& =\frac{n(n-1)}{(n-i-1)(n-i)} \frac{\left(2(n-i-2)\right)! ((n-2)!)^2 2^{2i}}{((n-i-2)!)^2 (2(n-2))!} \\
\end{alignat*}

Since each $i$ can be written as $i= \lfloor n \beta \rfloor$ with $\beta<1$,  $n-i$ goes to infinity when $n$ goes to infinity  and we can apply the Stirling formula:

\begin{alignat*}{1}
\frac{(2(n-i-2))! }{((n-i-2)!)^2 } & \underset{n\to+\infty}{\sim } \frac{\sqrt{2(n-i-2) 2 \pi} \left( \frac{2(n-i-2)}{e}\right)^{2(n-i-2)}}{2 \pi (n-i-2) \left(\frac{n-i-2}{e}\right)^{n-i-2}} \underset{n\to+\infty}{\sim} \frac{2^{2(n-i-2)}}{\sqrt{\pi(n-i-2)}}.
\end{alignat*}

Likewise, 

\begin{alignat*}{1}
\frac{(2(n-2))! }{((n-2)!)^2 } & \underset{n\to+\infty}{\sim} \frac{2^{2(n-2)}}{\sqrt{\pi(n-2)}}.
\end{alignat*}

Then, when $n$ goes to infinity

$$ \frac{B\left(i, n-i-\frac{5}{2} +1\right)}{B(i, n+1-i)}\underset{n\to+\infty}{\sim} \frac{1}{(1- \frac{i}{n+1})^{\frac{5}{2}}}.$$

Hence, we obtain 

$$ I_i^n =\frac{B\left(i, n-i-\frac{5}{2} +1\right)}{B(i, n+1-i)} E_i^n \underset{n\to+\infty}{\sim} \frac{1}{n} \frac{\frac{i}{n+1}}{(1- \frac{i}{n+1})^\frac{3}{2}}. $$

Finally, 
$$I_i^n \leq  \frac{2}{n} \frac{\frac{i}{n+1}}{(1- \frac{i}{n+1})^\frac{3}{2}}.$$ 

and 

\begin{alignat*}{1}
\frac{\eta}{ \sqrt{n}} \displaystyle\sum_{i= \lfloor n \alpha \rfloor}^{n-2}  \mathbb{E} \left( \left|L(U_{(i)}) \right| \frac{\left(U_{(i)} - \frac{i}{n+1}\right)^2}{2} \right) 
& \leq \frac{\eta}{2 \sqrt{n}} \displaystyle\sum_{i=\lfloor n \alpha \rfloor}^{n-2}  \mathbb{E} \left( \frac{(U_{(i)} - \frac{i}{n+1})^2}{\left(U_{(i)} - 1\right)^\frac{5}{2}} \right) \\
& \leq \frac{\eta}{ \sqrt{n}} \frac{1}{n} \displaystyle\sum_{i=\lfloor n \alpha \rfloor}^{n-2} \frac{\frac{i}{n+1}}{\left(1- \frac{i}{n+1}\right)^\frac{3}{2}} \\ & \sim \eta \frac{1}{n^{\frac{3}{2}}} \displaystyle\sum_{i=\lfloor n \alpha \rfloor}^{n-2} \frac{1}{\left(1- \frac{i}{n} \right)^{\frac{3}{2}}}.\\
\end{alignat*}
%

Finally, $R_n'$ converges in $L^1$ to 0 and so in probability. Hence, $R_n$ converges to 0 in probability.

\vspace{0.3cm}

\textbf{Step 3 : Identification of the limit}

\vspace{0.3cm}

Our goal is to find the limit of $\sqrt{n}Q_{n}$. Let us reorganize the expression of $Q_{n}$ to have a more classical form (sum of independent random variables) and to allow the use of the Proposition \ref{corlinde}.

Denoting by $$ \bar{Y}= \frac{\displaystyle\sum_{j=1}^{n+1}Y_{j}}{n+1},$$ we have thanks to (\ref{representation})

\begin{small}
\begin{equation*}
\begin{aligned}
Q_{n} &=  \frac{1}{n}\displaystyle\sum_{i=\lfloor n \alpha \rfloor +1}^{n} \left( \frac{\displaystyle\sum_{j=1}^{i}Y_{j}}{\displaystyle\sum_{k=1}^{n+1}Y_{k}} - \frac{i}{n+1} \right) l\left(\frac{i}{n+1}\right) \\
& = \frac{1}{n}\displaystyle\sum_{j=1}^{n} \left[  \left( \frac{Y_{j}}{\displaystyle\sum_{k=1}^{n+1}Y_{k}} - \frac{1}{n+1} \right) \displaystyle\sum_{i=\sup(\lfloor n \alpha \rfloor +1, j)}^n l\left(\frac{i}{n+1}\right) \right] \\
&= \frac{n+1}{\displaystyle\sum_{k=1}^{n+1}Y_{k}} \frac{1}{n(n+1)} \left[    \sum_{j=1}^{\lfloor n \alpha \rfloor +1} \left((Y_{j} - \bar{Y}) \displaystyle\sum_{i= \lfloor n \alpha \rfloor +1}^{n} l\left(\frac{i}{n+1}\right) \right) +  \displaystyle\sum_{j= \lfloor n \alpha \rfloor +2}^{n} \left( (Y_{j}- \bar{Y}) \displaystyle\sum_{i=j}^{n} l\left(\frac{i}{n+1}\right) \right) \right]. \\
\end{aligned} 
\end{equation*}
\end{small}

where we have permuted the two sums. The law of large numbers gives that $\bar{Y}$ converges a.s to 1 when $n$ goes to infinity. Then, thanks to Slutsky's lemma, we only need to study

 $$  \frac{1}{n(n+1)} \left[    \sum_{j=1}^{\lfloor n \alpha \rfloor +1} \left((Y_{j} - \bar{Y}) \displaystyle\sum_{i= \lfloor n \alpha \rfloor +1}^{n} l\left(\frac{i}{n+1}\right) \right) +  \displaystyle\sum_{j= \lfloor n \alpha \rfloor +2}^{n} \left( (Y_{j}- \bar{Y}) \displaystyle\sum_{i=j}^{n} l\left(\frac{i}{n+1}\right) \right) \right].$$

We set $\forall j \leq n$, $G^{n}_{j}:= \displaystyle\sum_{i=j}^{n} l\left(\frac{i}{n+1}\right),$  $G^{n}_{n+1}:=0$, $H^{n}: = \displaystyle\sum_{j= \lfloor n \alpha \rfloor +2}^{n} G_{j}$. Then 
 
 \begin{small}
 \begin{equation*}
 \begin{aligned}
Q_{n} & =  \frac{1}{n(n+1)} \left[ \displaystyle\sum_{j=1}^{\lfloor n \alpha \rfloor +1} \left( \frac{(n - \lfloor n \alpha \rfloor )G^{n}_{ \lfloor n \alpha \rfloor +1}- H^{n}}{n+1}  \right) Y_{j} 
 + \displaystyle\sum_{\lfloor n \alpha \rfloor +2}^{n+1} \left(G^{n}_{j} - \frac{H^{n}}{n+1} + G^{n}_{\lfloor n \alpha \rfloor +1} \frac{-1 - \lfloor n \alpha \rfloor}{n+1} \right) Y_{j} \right] \\
 & =\frac{1}{n(n+1)} \Bigg[ \displaystyle\sum_{j=1}^{\lfloor n \alpha \rfloor+1} \left( \frac{(n - \lfloor n \alpha \rfloor)G^{n}_{ \lfloor n \alpha \rfloor +1}- H^{n}}{n+1}  \right) (Y_{j}-1) \\
 &+ \displaystyle\sum_{\lfloor n \alpha \rfloor +2}^{n+1} \left(G^{n}_{j} - \frac{H^{n}}{n+1} + G^{n}_{\lfloor n \alpha \rfloor +1} \frac{-1 - \lfloor n \alpha \rfloor}{n+1} \right) (Y_{j}-1) \Bigg], \\
\end{aligned}
\end{equation*}
 \end{small}

because 

  \begin{small}
 
 \begin{alignat*}{1}
& \frac{1}{n(n+1)} \left[ \displaystyle\sum_{j=1}^{\lfloor n \alpha \rfloor +1} \left( \frac{n - \lfloor n \alpha \rfloor)G^{n}_{ \lfloor n \alpha \rfloor +1}- H^{n}}{n+1}  \right)+ \displaystyle\sum_{\lfloor n \alpha \rfloor +2}^{n+1} \left(G^{n}_{j} - \frac{H^{n}}{n+1} + G^{n}_{\lfloor n \alpha \rfloor +1} \frac{-1 - \lfloor n \alpha \rfloor}{n+1} \right) \right] \\
& =\frac{1}{n(n+1)} \Bigg[ \left( \frac{G^{n}_{\lfloor n \alpha \rfloor} (n- \lfloor n \alpha \rfloor) -H^{n} }{n+1}\right)(\lfloor n \alpha \rfloor +1) \\
&+ \left( \frac{G^{n}_{\lfloor n \alpha \rfloor} (-1 - \lfloor n \alpha \rfloor)}{n+1}(n - \lfloor n \alpha \rfloor ) - \frac{H^{n}}{n+1}(n- \lfloor n \alpha \rfloor) + H^{n}\right) \Bigg]\\
& = 0.
 \end{alignat*}
 
 \end{small}

Finally, we obtain
 
 \begin{small}
 \begin{equation*}
 \begin{aligned}
Q_{n} & =
\frac{1}{n(n+1)} \Bigg[ \displaystyle\sum_{j=1}^{\lfloor n \alpha \rfloor +1} \left( \frac{(n - \lfloor n \alpha \rfloor)G^{n}_{ \lfloor n \alpha \rfloor +1}- H^{n}}{n+1}  \right) (Y_{j}-1) + \displaystyle\sum_{\lfloor n \alpha \rfloor +2}^{n+1} \left(G^{n}_{j} - \frac{H^{n}}{n+1} + G^{n}_{\lfloor n \alpha \rfloor +1} \frac{-1 - \lfloor n \alpha \rfloor}{n+1} \right) (Y_{j}-1) \Bigg] \\
 & = \frac{1}{n+1} \displaystyle\sum_{j=1}^{n+1} \alpha_{j, n} (Y_{j}-1),\\
\end{aligned}
\end{equation*}
\end{small}

where 

$$\alpha_{j, n}= \left( \frac{(n - \lfloor n \alpha \rfloor)G^{n}_{ \lfloor n \alpha \rfloor +1}- H^{n}}{n(n+1)}  \right), \, \, \forall j \leq \lfloor n \alpha \rfloor +1$$

and 

$$\alpha_{j, n}= \left( \frac{G^{n}_{j}(n+1) - H^{n} - G^{n}_{\lfloor n \alpha \rfloor +1} (1 + \lfloor n \alpha \rfloor)}{n(n+1)} \right) , \, \, \forall j \geq  \lfloor n \alpha \rfloor + 2.$$

Let us check the assumptions of the Proposition \ref{corlinde}. To begin with, let us show that $\sigma_{n}^{2}$ converges. We have 

\begin{alignat*}{1}
\sigma_{n}^{2} & =\frac{1}{n+1} \displaystyle\sum_{j=1}^{n+1} \alpha_{j, n}^{2} \\
&= \frac{\lfloor n \alpha \rfloor +1}{n+1} \left[\frac{(n - \lfloor n \alpha \rfloor)G^{n}_{ \lfloor n \alpha \rfloor +1}- H^{n}}{n(n+1)}\right]^{2}  \\
& +  \frac{1}{n+1} \displaystyle\sum_{j=\lfloor n \alpha \rfloor +2}^{n+1}\left( \frac{G^{n}_{j}(n+1)}{n(n+1)} \right)^{2} +2\frac{1}{n+1} \displaystyle\sum_{j=\lfloor n \alpha \rfloor +2}^{n+1}  \frac{G^{n}_{j}(n+1)( - H^{n} - G^{n}_{\lfloor n \alpha \rfloor +1} (1 + \lfloor n \alpha \rfloor))}{n^{2}(n+1)^{2}} \\ 
& +  \frac{1}{n+1} \displaystyle\sum_{j=\lfloor n \alpha \rfloor +2}^{n+1}\left(\frac{- H^{n} - G^{n}_{\lfloor n \alpha \rfloor +1} (1 + \lfloor n \alpha \rfloor)}{n(n+1)} \right)^{2}. \\
\end{alignat*}

Let us work with the two terms which depend on $G^{n}_{j}$. The first term can be expanded as

\begin{small}

\begin{alignat*}{1}
 \frac{1}{n+1} \displaystyle\sum_{j=\lfloor n \alpha \rfloor +2}^{n+1}\left( \frac{G^{n}_{j}(n+1)}{n(n+1)} \right)^{2} & = \frac{1}{n^{2}(n+1)} \displaystyle\sum_{j=\lfloor n \alpha \rfloor +2}^{n+1}\Big( G^{n}_{j} \Big)^{2} \\
 & =  \frac{1}{n^{2}(n+1)} \displaystyle\sum_{j=\lfloor n \alpha \rfloor +2}^{n+1} \left( \displaystyle\sum_{i=j}^{n} l\left(\frac{i}{n+1}\right)\right)^{2} \\
 & = \frac{1}{n(n+1)} \displaystyle\sum_{i_{1}=\lfloor n \alpha \rfloor +2}^{n+1} \displaystyle\sum_{i_{2}=\lfloor n \alpha \rfloor +2}^{n+1} \left(\frac{-1 + (i_{1} \wedge i_{2}) - \lfloor n \alpha \rfloor}{n+1}\right)l\left(\frac{i_{1}}{n+1}\right)l\left(\frac{i_{2}}{n+1}\right). \\
\end{alignat*}

\end{small}

The second term may be rewritten as 

\begin{small}
$$ 2\frac{1}{n+1} \displaystyle\sum_{j=\lfloor n \alpha \rfloor +2}^{n+1}  \frac{G^{n}_{j}(n+1)( - H^{n} - G^{n}_{\lfloor n \alpha \rfloor +1} (1 + \lfloor n \alpha \rfloor))}{n^{2}(n+1)^{2}}  
= -2\frac{(H^{n})^{2}}{n^{2}(n+1)^{2}} - 2 \frac{(1 + \lfloor n \alpha \rfloor )}{(n+1)^{2}n^{2}} H^{n} G^{n}_{ \lfloor n \alpha \rfloor +1}.$$
\end{small}

Finally,

\begin{alignat*}{1}
\sigma_{n}^{2}& = \frac{\lfloor n \alpha \rfloor +1}{n+1} \left[\frac{(n - \lfloor n \alpha \rfloor)G^{n}_{ \lfloor n \alpha \rfloor +1}- H^{n}}{n(n+1)}\right]^{2}  \\
& +  \frac{1}{n(n+1)} \displaystyle\sum_{i_{1}=\lfloor n \alpha \rfloor +2}^{n+1} \displaystyle\sum_{i_{2}=\lfloor n \alpha \rfloor +2}^{n+1} \left(\frac{-1}{n+1}
 + \frac{\min(i_{1}, i_{2})}{n+1} - \frac{\lfloor n \alpha \rfloor}{n+1} \right) l(\frac{i_{1}}{n+1})l(\frac{i_{2}}{n+1}) \\
 & -2\frac{(H^{n})^{2}}{n^{2}(n+1)^{2}} - 2 \frac{(1 + \lfloor n \alpha \rfloor )}{(n+1)^{3}n^{2}} H^{n} G^{n}_{ \lfloor n \alpha \rfloor +1} +  \frac{n - \lfloor n \alpha \rfloor -1 }{n+1} \left(\frac{ H^{n} + G^{n}_{\lfloor n \alpha \rfloor +1} (1 + \lfloor n \alpha \rfloor)}{n(n+1)} \right)^{2}. \\
 \end{alignat*}
 
 Let us first notice that, if we denote  $$K^{n}= \displaystyle\sum_{i_{1}=\lfloor n \alpha \rfloor +2}^{n+1} \displaystyle\sum_{i_{2}=\lfloor n \alpha \rfloor +2}^{n+1} \frac{\min(i_{1}, i_{2})}{n+1}  l\left(\frac{i_{1}}{n+1}\right)l\left(\frac{i_{2}}{n+1}\right) $$ and $$T^{n} = \displaystyle\sum_{i =\lfloor n \alpha \rfloor }^{n} \frac{i}{n} l\left( \frac{i}{n+1}\right)$$ then
 
 $$ H^{n} = nT^{n} - (\lfloor n \alpha \rfloor +1) G^{n}_{ \lfloor n \alpha \rfloor +1}.$$
 
So that 
\begin{equation*}
\begin{aligned}
\sigma_{n}^{2} & \underset{n\to+\infty}{\sim} \alpha\frac{(G^{n}_{ \lfloor n \alpha \rfloor +1}-T^{n})^{2}}{n^{2}} + \frac{K^{n} - \alpha (G^{n}_{ \lfloor n \alpha \rfloor +1})^{2}}{n^{2}} - \frac{-2(T^{n} - \alpha G^{n}_{ \lfloor n \alpha \rfloor +1})^{2}}{n^{2}}\\
& -2 \frac{\alpha(G^{n_{ \lfloor n \alpha \rfloor +1}}T^{n}- \alpha (G^{n}_{ \lfloor n \alpha \rfloor +1})^{2})}{n^{2}} + \frac{(1- \alpha)(T^{n})^{2}}{n^{2}} \\
& \underset{n\to+\infty}{\sim} \frac{K^{n}-( T^{n})^{2}}{n^{2}}. \\
\end{aligned}
\end{equation*}

Let us show that this last quantity converges to $\sigma^2 = \int_{\alpha}^1 \int_{\alpha}^1 \frac{\min(x, y) -xy}{f(F^{-1}(x)) f(F^{-1}(y))} < \infty$. Indeed it is a generalized Rieman sum. First, we show that the function 
$$g : (x, y) \mapsto \frac{\min(x, y) -xy}{f(F^{-1}(x)) f(F^{-1}(y))}$$ is integrable on $]\alpha, 1[ \times ]\alpha, 1[$. Indeed, around 1, 

$$g(x, y) = O \left(\frac{\min(x, y) -xy}{(1-x)^{\frac{3}{2}- \epsilon_l}(1-y)^{\frac{3}{2}- \epsilon_l}} \right)$$

which is integrable on this domain because for $\beta$ close to 1

$$\int_{\alpha}^{\beta} \int_{\alpha}^{\beta} \frac{\min(x, y) -xy}{(1-x)^{\frac{3}{2}- \epsilon_l}(1-y)^{\frac{3}{2}- \epsilon_l}} dxdy \sim C(\alpha) \beta(1- \beta)^{2 \epsilon_l}.$$

and $\epsilon_l>0$ (here again we see that we really need this $\epsilon_L>0$).

As we have already seen, the results on Riemann's sum in dimension 2, give by the continuity of the function $(x, y) \mapsto \frac{\min(x, y) -xy}{f(F^{-1}(x)) f(F^{-1}(y))}$ that for all $\alpha <\beta <1$ :

$$
\sigma_{n, \beta}^2  := \frac{1}{n^2}\sum_{i_1=\lfloor n \alpha \rfloor}^{\lfloor n \beta \rfloor} \sum_{i_1=\lfloor n \alpha \rfloor}^{\lfloor n \beta \rfloor} \frac{\frac{\min(i_1, i_2)}{n} - \frac{i_1i_2}{n^2} }{f\left(F^{-1}\left(\frac{i_1}{n+1}\right)\right) f\left(F^{-1}\left(\frac{i_2}{n+1}\right)\right)} 
 \longrightarrow \int_{\alpha}^{\beta} \int_{\alpha}^{\beta} \frac{\min(x, y) -xy}{f(F^{-1}(x)) f(F^{-1}(y))} dxdy. 
$$

We have to study the remaining part of the sum to conclude. Let us fix $\beta$ close to 1 and deal with

$$r_{n, \beta}^2  := \frac{1}{n^2}\sum_{i_1=\lfloor n \beta \rfloor}^{n} \sum_{i_1=\lfloor n \beta \rfloor}^{n} \frac{\frac{\min(i_1, i_2)}{n} - \frac{i_1i_2}{n^2} }{f \left(F^{-1}\left(\frac{i_1}{n+1}\right)\right) f\left(F^{-1}\left(\frac{i_2}{n+1}\right)\right)}.$$

First of all, let us notice that

$$r_{n, \beta}^2= \int_{\beta}^1 \int_{\beta}^1 g\left(\frac{\lfloor n x \rfloor}{n},\frac{\lfloor n y \rfloor}{n} \right) dx dy. $$

We want to show that $$\lim\limits_{n \to + \infty} r_{n, \beta}^2= \int_{\beta}^1\int_{\beta}^1 g(x, y) dx dy.$$ 

Let us first show that, using Lebesgue theorem, we can permute the limit in $n$ and the double integrale, in this way

$$\lim\limits_{n \to + \infty} r_{n, \beta}^2 = \int_{\beta}^1\int_{\beta}^1 \lim\limits_{n \to + \infty}g\left(\frac{\lfloor n x \rfloor}{n},\frac{\lfloor n y \rfloor}{n} \right).$$

\begin{itemize}
\item[1)] Let $(x, y)$ be fixed in $[\beta, 1[ \times [\beta, 1[$ and $n$. Then $$g\left(\frac{\lfloor n x \rfloor}{n},\frac{\lfloor n y \rfloor}{n} \right) \longrightarrow g(x, y)$$ by continuity. And $g$ is integrable on $[\beta, 1[ \times [\beta, 1[$ as we saw before. 
\item[2)] Let $(x, y)$ be fixed in $[\beta, 1[ \times [\beta, 1[$ and $n$. Let us denote $x_n = \frac{\lfloor n x \rfloor}{n}$ and $y_n=\frac{\lfloor n y \rfloor}{n}$. By hypothesis

$$g\left(\frac{\lfloor n x \rfloor}{n},\frac{\lfloor n y \rfloor}{n} \right) \leq C \frac{\min(x_n, y_n) -x_ny_n}{(1-x_n)^{\frac{3}{2}- \epsilon_l}(1-y_n)^{\frac{3}{2}- \epsilon_l}}.$$

By separating the two cases and using monotony we obtain that

$$g\left(\frac{\lfloor n x \rfloor}{n},\frac{\lfloor n y \rfloor}{n} \right) \leq C h(x, y)$$

where $$h:(x, y) \mapsto \frac{\min(x, y)}{\left(1- \min(x, y)\right)^{\frac{3}{2}- \epsilon_l}\left(1- \max(x, y)^{\frac{3}{2}- \epsilon_l-1}\right)}$$

is integrable on $[\beta, 1[ \times [\beta, 1[$. 

\end{itemize}

Then, the Lebesgue theorem allows us to permute integration and limit so that, we have shown that $\sigma_n^2 \longrightarrow \sigma^2$and the first assumption of Proposition \ref{corlinde} holds. 

\vspace{0.3cm}
Let us now deal with the second assumption of Proposotion \ref{corlinde} about the maximum of the $\alpha_{i, n}$. 

For $j \leq \lfloor n \alpha \rfloor +1$, we have 
 
 \begin{alignat*}{1}
 \alpha_{j, n} = & \frac{(n - \lfloor n \alpha \rfloor +1)G^{n}_{ \lfloor n \alpha \rfloor +1} - H^{n}}{n(n+1)}. \\
  \end{alignat*}
  
Using the previous computations, for $n$ large enough we have
 $$\frac{ (\alpha_{j, n})^{2} }{n \sigma^{2}_{n}} \underset{n \to + \infty}{\sim} \frac{(G^{n}_{ \lfloor n \alpha \rfloor +1} - \frac{T^{n}}{n})^{2}}{K_{n} - \frac{T_{n}^{2}}{n^{2}}} \frac{1}{n}, $$

But the convergence $$ \frac{\left(K^{n} - \left(\frac{T^{n}}{n}\right)\right)^{2}}{n^{4} } \underset{n \to + \infty}{\longrightarrow} \int_{\alpha}^{1} \int_{\alpha}^{1} (\min(x, y) - xy) l(x) l(y) dxdy  $$ implies the convergence $$\frac{(G^{n}_{ \lfloor n \alpha \rfloor +1} - \frac{T^{n}}{n})^{2}}{n^{4}} \underset{n \to + \infty}{\longrightarrow} \int_{\alpha}^{1} (1-x) l(x) dx.$$ Indeed 
 
$$\int_{\alpha}^{1} \int_{\alpha}^{1} (\min(x, y) -xy)l(x) l(y) dx dy = \int_{\alpha}^{1} \int_{\alpha}^{1} (y(1-x)) l(x) l(y) dx dy + \int_{\alpha}^{1} x l(x) \int_{x}^{1} (1-y)l(y) dy dx.$$
 
So that,
 
 $$\frac{ (\alpha_{j, n})^{2} }{n \sigma^{2}_{n}} \underset{n \to + \infty}{\sim} \frac{C}{n} \underset{n \to + \infty}{\longrightarrow} 0$$

when $n$ goes to infinity. If  $j \geq \lfloor n \alpha \rfloor +2 $ the same property holds as
 
 $$\alpha_{j, n}= \frac{(n+1)G_{j}^{n}-H^{n} - G^{n}_{\lfloor n \alpha \rfloor}( \lfloor n \alpha \rfloor +1) }{n(n+1)} \underset{n\to+\infty}{\sim} \frac{(n+1) G_{j^{n}} - T^{n}}{n^{2}}. $$
 
Hence, we may apply Proposition \ref{corlinde} and conclude that
 
 $$\sqrt{n}Q_{n} \Longrightarrow \mathcal{N}(O, \sigma^{2}).$$
 
where $\sigma^{2} = \int_{\alpha}^{1} \int_{\alpha}^{1} (\min(x, y) -xy) l(x) l(y) dxdy$. Finally, just multiply by $(1-\alpha)^{-1}$ to get the final result.

\vspace{0.3cm}

\textbf{Step 4 : Conclusion}

The Slutsky lemma allows to conclude using the results of steps 1 and 3.

\end{proof}

\section{Conclusion}

The superquantile was introduced because the usual quantile was not subadditive and does not give enough information on what was happening in the tail-distribution. This quantity is interesting because it satisfies the axioms of a coherent measure of risk. In this paper, we have introduced a new coherent measure of risk with the help of the Bregman divergence associated to a strictly convex function $\gamma$. They are rich tools because of the diversity of the functions $\gamma$ that can be chosen according to the problem we study. We shown throught different examples that a judicious choice for the function $\gamma$ can make the Bregman superquantile a more interesting measure of risk than the classical superquantile (Bregman superquantile can be finite even in infinite mean model, it is continous in more cases, more \textit{robust} etc...).  Moreover, we have introduced a Monte Carlo estimator of the Bregman superquantile which is statistically powerful thanks to the strictly convex (and so settling) function $\gamma$. 

The theoretical properties obtained in this paper are confirmed on several numerical test cases.
More precisely, geometrical and harmonic superquantiles are more robust than the classical superquantile.
This robustness is particularly important in in finance and risk assessment studies.
For instance, in risk assessment, when dealing with real data, geometrical and harmonic statistics have been proved to be more relevant than classical statistics.
For example \cite{chahan04} prove the usefulness of the geometrical mean and variance for the analysis of air quality measurements.
As an illustration, we have applied the geometrical and harmonic superquantiles on real data coming from a radiological impact code used in the nuclear industry.

Further studies will try to apply these criteria in probabilistic assessment of physical components reliability using numerical simulation codes \cite{derdev08}.
However,  Monte Carlo estimators are no longer applicable in this context and efficient estimators have to be developed.
Ideas involving response surface technique should be developed (see for example \cite{cangar08} for quantile estimation and \cite{becgin12} for rare event probability estimation).

Finally those Bregman superquantiles are interesting because they can be linked with several previous works in economy. There is for example a strong link between capital allocations and Bregman superquantile. It would be interesting to apply our results on this economic theory in a future work. 
\newpage

\begin{center}
\textbf{ACKNOWLEDGEMENT}
\end{center}

We warmly thank two anonymous referees for their careful reading that helped us to improve the paper.

\bibliographystyle{plain}
\bibliography{bibi}
\nocite{*}

\end{document}